\documentclass[a4paper,12pt]{amsart}
\usepackage[english]{babel}
\usepackage[T1]{fontenc}
\usepackage{array}
\usepackage{makecell}
\usepackage[a4paper,margin=2.5cm]{geometry}
\usepackage[justification=centering]{caption}

\usepackage{enumerate}
\usepackage{tabularx} 
\usepackage{amsmath}  
\usepackage{graphicx} 

\usepackage{amsmath,amscd,amsbsy,amssymb,latexsym,url,bm,amsthm,mathrsfs}
\usepackage{epsfig,graphicx,subfigure}
\usepackage{enumerate,balance,mathtools}
\usepackage{wrapfig}
\usepackage{mathrsfs,euscript}
\usepackage[vlined,boxed,commentsnumbered,linesnumbered,ruled]{algorithm2e}
\usepackage{amsfonts}
\usepackage{amsthm}
\usepackage{fancyhdr}
\usepackage{picinpar}
\usepackage{listings}
\usepackage{layout}
\usepackage[all]{xy}
\usepackage{indentfirst}
\usepackage{tikz-cd}
\usepackage{tikz}
\usepackage[all]{xy}
\usepackage{mathrsfs}
\usepackage[colorlinks=true, allcolors=blue]{hyperref}
\usepackage{setspace}

\usepackage{amsmath,amsthm,amssymb,amsxtra,calligra,mathrsfs}
\usepackage{graphicx}
\usepackage{tikz}
\usepackage{tikz-cd} 
\usepackage{xcolor}
\usepackage{upgreek}
\usepackage{comment}
\usepackage{graphicx}
\usepackage{mathtools}
\usepackage{extarrows}
\usepackage{faktor}
\usepackage{mathrsfs,euscript}
\usepackage{comment}
\usepackage{bookmark}
\usepackage{hyperref}
\usepackage{mathrsfs}
\usepackage{multirow}

\hypersetup{
    colorlinks=true, 
    citecolor=blue,
    linkcolor=magenta,
    pdfauthor={},
	bookmarksnumbered=true,
}

\newcommand{\wh}[1]{\widehat{#1}}

\newcommand{\wt}[1]{\widetilde{#1}}

\newcommand{\mb}[1]{\mathbb{#1}}

\newcommand{\ove}[1]{\overline{#1}}

\newcommand{\mtc}[1]{\mathcal{#1}}
\newcommand{\mtf}[1]{\mathfrak{#1}}
\newcommand{\mts}[1]{\mathscr{#1}}

\newcommand{\cQ}{\mathcal{Q}}

\DeclareMathOperator{\coeff}{coeff}
\DeclareMathOperator{\Fut}{Fut}

\DeclareMathOperator{\PGL}{PGL}
\DeclareMathOperator{\SO}{SO}

\DeclareMathOperator{\Spec}{Spec}
\DeclareMathOperator{\Proj}{Proj}

\DeclareMathOperator{\mult}{mult}

\DeclareMathOperator{\Pic}{Pic}
\DeclareMathOperator{\Hom}{Hom}

\DeclareMathOperator{\ord}{ord}

\DeclareMathOperator{\Ext}{Ext}

\DeclareMathOperator{\codim}{codim}

\DeclareMathOperator{\edim}{edim}

\DeclareMathOperator{\rk}{rank}

\DeclareMathOperator{\Bs}{Bs}

\DeclareMathOperator{\sst}{ss}
\DeclareMathOperator{\Id}{Id}
\DeclareMathOperator{\NS}{NS}

\DeclareMathOperator{\che}{ch}

\DeclareMathOperator{\GIT}{GIT}

\DeclareMathOperator{\nodal}{n}
\DeclareMathOperator{\Aut}{Aut}

\DeclareMathOperator{\vol}{vol}

\DeclareMathOperator{\CM}{CM}
\DeclareMathOperator{\Bl}{Bl}
\DeclareMathOperator{\Hodge}{Hodge}

\newcommand{\sslash}{\mathbin{\mkern-3mu/\mkern-6mu/\mkern-3mu}}

\newcommand{\sheafHom}{\mathscr{H}\text{\kern -3pt {\calligra\large om}}\,}

\DeclareMathOperator{\Hilb}{\mathrm{Hilb}}

\newcommand{\bA}{{\mathbb A}}
\newcommand{\bZ}{{\mathbb Z}}
\newcommand{\bN}{{\mathbb N}}
\newcommand{\bC}{{\mathbb C}}
\newcommand{\bQ}{{\mathbb Q}}
\newcommand{\bP}{{\mathbb P}}
\newcommand{\bR}{{\mathbb R}}

\newcommand{\bV}{{\mathbb V}}

\usepackage{dutchcal}

\newtheorem{theorem}{Theorem}[section]
\newtheorem{theodef}[theorem]{Theorem-Definition}
\newtheorem{lemma}[theorem]{Lemma}

\newtheorem{corollary}[theorem]{Corollary}
\newtheorem{prop}[theorem]{Proposition}

\newtheorem{remark}[theorem]{Remark}

\theoremstyle{definition}
\newtheorem{defn}[theorem]{Definition}

\theoremstyle{remark}

\title{K-moduli of Fano threefolds and genus four curves}
\date{\today}
\author{Yuchen Liu}
\address{Department of Mathematics, Northwestern University, 2033 Sheridan Rd, Evanston, Illinois 60208, USA}
\email{yuchenl@northwestern.edu}
\author{Junyan Zhao}
\address{Department of Mathematics, Statistics, and Computer Science, University of Illinois Chicago, 851 S Morgan St, Chicago, Illinois 60607, USA}
\email{jzhao81@umd.edu}

\begin{document}

\maketitle

\begin{abstract}
In this article, we study the K-moduli space of Fano threefolds obtained by blowing up $\mb{P}^3$ along $(2,3)$-complete intersection curves. This K-moduli space is a two-step birational modification of the GIT moduli space of bidegree $(3,3)$-curves on $\bP^1\times \bP^1$. As an application, we show that our K-moduli space appears as one model of the Hassett--Keel program for $\overline{M}_4$. In particular, we classify all K-(semi/poly)stable members in this deformation family of Fano varieties. We follow the moduli continuity method with moduli of lattice-polarized K3 surfaces, general elephants and Sarkisov links as new ingredients.



\end{abstract}

\tableofcontents

\section{Introduction}

K-stability is an algebro-geometric condition introduced by differential geometers \cite{Tia97, Don02} to characterize the existence of K\"{a}hler--Einstein metrics on Fano varieties. Recently, the algebraic theory of K-stability has seen spectacular development. One important outcome of this study is the construction of the \emph{K-moduli stack}, parametrizing the K-semistable Fano varieties of given dimension and volume. Moreover, the stack admits a projective good moduli space, known as the \emph{K-moduli space}, that parameterizes K-polystable Fano varieties. See \cite{Xu21, Xu24} for recent surveys on this topic. 

Despite the general theory being completed, it remains a challenge to find explicit descriptions of K-moduli spaces for smooth Fano varieties and their limits. A notable approach to address this challenge, often called the \emph{moduli continuity method}, can be summarized as follows. It is usually easy to find one  K-stable member in a given family of Fano varieties using equivariant K-stability \cite{Zhu21}. This implies that a general member of this family is K-stable by openness of K-(semi)stability \cite{BL18, BLX19, Xu20}.  Then by the ampleness of the CM line bundle \cite{CP21, XZ20} one identifies a candidate of the K-moduli space from GIT stability for the CM line bundle on a suitable parameter space. Then one needs to understand the geometry of K-semistable limits to show that it appears in the GIT parameter space, where a crucial ingredient is an a priori estimate for singularities from \cite{Liu18} (after \cite{Fuj18}).  Finally, one concludes the isomorphism between the K-moduli space and GIT moduli space using separatedness and properness of K-moduli spaces \cite{BX19, ABHLX20, LXZ22}. This method was successful in dimension two \cite{MM93, OSS16}, and for several families in higher dimensions: \cite{LX19} on cubic threefolds, \cite{SS17} on quartic del Pezzo manifolds, \cite{ADL22} on quartic double solids, and \cite{Liu22} on cubic fourfolds.


In this paper, we study the K-moduli space of Fano threefolds of the family \textnumero 2.15 from the Iskovskikh--Mori--Mukai classification \cite{IP99, MM03}, following the moduli continuity method. A general member in this family is the blow-up of $\bP^3$ along a smooth $(2,3)$-complete intersection curve, which is shown to be K-stable (cf. \cite[Proposition 4.33]{CA21}). This K-moduli space is fascinating as it provides a birational model of the moduli space $\ove{M}_4$ of genus four curves. 

Our first main result shows that every member in this K-moduli space is indeed isomorphic to the blow-up of $\mb{P}^3$ along a $(2,3)$-complete intersection curve. Moreover, we identify the K-moduli space of Fano threefolds of family \textnumero 2.15 with a model in the variation of GIT moduli spaces $\ove{M}^{\GIT}(t)$ (cf. Definition \ref{gitdefn} and Remark \ref{vgitremark}(2)) of $(2,3)$-complete intersection curves studied by Casalaina-Martin--Jensen--Laza (cf. \cite{CMJL14}). In particular, we show that our K-moduli space is a two-step birational modification (a blow-up followed by a flip) of the GIT moduli space of $(3,3)$-curves on $\bP^1\times \bP^1$.

\begin{theorem}\label{Main1}
 Let $\mts{M}^K_{\textup{№2.15}}$ be the K-moduli stack of the family №2.15, and $\ove{M}^K_{\textup{№2.15}}$ be its good moduli space. Then the followings hold.
\begin{enumerate}
    \item Every K-semistable Fano variety in $\mts{M}^K_{\textup{№2.15}}$ is isomorphic to $\Bl_C\mb{P}^3$ for some $(2,3)$-complete intersection $C$.
    \item The K-moduli stack $\mts{M}^K_{\textup{№2.15}}$ is a smooth connected component of $\mts{M}^K_{3,22}$, and the K-moduli space $\ove{M}^K_{\textup{№2.15}}$ is a  normal projective variety.
    \item Moreover,  there is a natural isomorphism between $\ove{M}^K_{\textup{№2.15}}$ and the VGIT moduli space $\ove{M}^{\GIT}(\frac{22}{51})=\ove{M}^{\GIT}(\frac{2}{5},\frac{1}{2})$, fitting into the following diagram
    \begin{equation}\nonumber
\xymatrix @R=.07in @C=.07in{
 &   & & \ove{M}^K_{\textup{№2.15}}  \ar[rdd]_{\iota} \ar@{-->}[rr]& & \ove{M}^{\GIT} ( \frac{2}{9} , \frac{2}{5} ) \ar[ldd]^{\psi} \ar[rdd]^{\rho} & \\
&&&&&& \\
 & &  & & \ove{M}^{\GIT} (\frac{2}{5} ) & & \ove{M}^{\GIT} ( 0 , \frac{2}{9} ) & \simeq \ |\mtc{O}_{\mb{P}^1\times\mb{P}^1}(3,3)|\sslash \Aut(\mb{P}^1\times\mb{P}^1),  \\
}
 \end{equation}
 where
 \begin{itemize}
     \item $\rho$ is a divisorial contraction of a birational transform of the Gieseker–Petri divisor to the point parametrizing the triple conic;
     \item $\psi$ is a flip of the locus parametrizing non-reduced curves; and
     \item $\iota$ is a flipping contraction of the locus parametrizing nodal curves whose normalization is hyperelliptic.
 \end{itemize}
\end{enumerate}
\end{theorem}

Our second result gives a complete description of the K-(semi/poly)stable members in this family, including their $\bQ$-Gorenstein degenerations. 

\begin{theorem}\label{Main2}
    Let $C$ be a $(2,3)$-complete intersection curve in $\mb{P}^3$. Let $Q$ be the unique quadric surface containing $C$, and $X:=\Bl_C\mb{P}^3$ be the blow-up of $\mb{P}^3$ along $C$. Then $X$ is 
    \begin{enumerate}
        \item K-semistable if and only if one of the followings holds:
        \begin{itemize}
            \item $Q$ is a smooth quadric, and the curve $C$ has at worst $A_n$- or $D_4$-singularities;
            \item $Q$ is a quadric cone, and $C$ has at worst $A_n$- or $D_4$-singularities in the smooth locus of $Q$, and has at worst an $A_1$-singularity at the cone point. 
        \end{itemize}
        \item K-stable if and only if one of the followings holds:
         \begin{itemize}
            \item $Q$ is a smooth quadric, $C$ has at worst $A_n$-singularities, and $C$ contains no line component $L$ meeting the residual curve $C'$ in exactly one point;
            \item $Q$ is a quadric cone, $C$ has at worst $A_n$-singularities in the smooth locus of $Q$, and has at worst an $A_1$-singularity at the cone point. 
        \end{itemize}
        \item K-polystable but not K-stable if and only if one of the followings hold:
            \begin{itemize}
            \item $Q$ is normal and the curve $C$ is a union of three conics meeting in two $D_4$-singularities;
             \item $Q$ is a smooth quadric, and $C=C_{2A_5}=\mb{V}(x_0x_3-x_1x_2,x_0x_2^2+x^2_1x_3)$ is the (unique) maximally degenerate curve with two $A_5$-singularities. 
        \end{itemize}
    \end{enumerate}
    In particular, every smooth member in this deformation family of Fano threefolds is K-stable.
\end{theorem}

Note that the last statement of Theorem \ref{Main2} was proved independently in \cite[Main Theorem]{GDGV23} using a completely different method. As a side remark, the approach in \cite{GDGV23} is computational, which applies the Abban--Zhuang method to the family \textnumero 2.15.

As a consequence of the above result and \cite[Theorem 7.1]{CMJL14}, we find that the K-moduli space coincides with one model of the Hassett--Keel program (cf. Subsection \ref{HKP}) for $\overline{M}_4$. 

\begin{theorem}\label{hkp}
    The K-moduli space $\ove{M}^K_{\textup{№2.15}}$ is isomorphic to the model $\ove{M}_4(\frac{1}{2},\frac{23}{44})$ for the Hassett--Keel program of genus four curves.
\end{theorem}

We note that some recent works on K-moduli spaces for various families of  Fano threefolds follow a different approach, namely the \emph{Abban--Zhuang method} \cite{AZ22}, based on estimates for the stability threshold using inversion of adjunction and multi-graded linear systems. See e.g. \cite{Pap22, ACD+23, CT23} when the dimension of the moduli is small, and \cite{ACKLP, CDG+, CFFK, DJKQ} for additional families.

\subsection*{Sketch of Proof}

Our main approach is the \emph{moduli continuity method}, which relies the properness of the K-moduli spaces.

The most technical aspect of proving Theorem \ref{Main1} lies in identifying every (singular) K-semistable $\bQ$-Fano variety which admits a $\bQ$-Gorenstein smoothing to smooth Fano №2.15 as the blow-up of $\mathbb{P}^3$ along a $(2,3)$-complete intersection curve. This is the content of Section \ref{K3SURFACES} and Section \ref{KSSLIMITS}. For any K-semistable member $X$ in the family, by the properness of the K-moduli space, we can choose a one-parameter family $f:\mts{X}\rightarrow T$ over a smooth pointed curve $0\in T$ such that $\mts{X}_0\simeq X$, with other fibers being K-stable and isomorphic to the blow-ups of $\mathbb{P}^3$ along smooth $(2,3)$-complete intersection curves. By applying local-to-global volume comparison (cf. \cite{Liu18, LX19}) and general elephant techniques, we establish that the anti-canonical divisor is Cartier and its linear series is base-point-free. In fact, we establish a more general result applicable to many families of Fano threefolds (cf. Theorem \ref{nonvanishing}), anticipating that our method can be extended to study the K-moduli of other families as well.

Next, we show that $X$ is also the blow-up of $\bP^3$ along a $(2,3)$-complete intersection curves (cf. Theorem \ref{23complete}). Since a general fiber $\mts{X}_t$ has Picard rank $2$, it admits a Sarkisov link structure 
$$\xymatrix{
& & \mts{X}_t \ar[dl]_{\pi_t} \ar[dr]^{\phi_t} &\\
 & \mb{P}^3\ar@{-->}[rr]   &  &  \mts{V}_t\subseteq \mb{P}^4,
 }$$
where $\pi_t$ is a blow-up along a $(2,3)$-complete intersection curve, and $\phi_t$ is a blow-up of a singular cubic threefold $\mts{V}_t$ at a double point. To prove Theorem \ref{23complete}, we show that the above Sarkisov link structures on general fibers degenerate to a Sarkisov link structure of the same form on the special fiber $\mts{X}_0$. This is a rather strong statement, as Sarkisov link structures on Fano varieties could easily break under degenerations, e.g. the isotrivial degeneration of $\bP^1\times \bP^1$ to $\bP(1,1,2)$ loses both conic bundle structures on the special fiber (see also \cite{DJKQ} for an example of K-semistable degenerations). Nevertheless, working with the Sarkisov link structure enables us to switch between different models. As it turns out, the second model of blow-up of cubic threefolds is better suited for the proof.

In order to show that Sarkisov link structures degenerate, we focus on the model of blow-ups of cubic threefolds. Denote by $\mtc{L}_t := \phi_t^* \mtc{O}_{\mts{V}_t}(1)$ and $\mts{Q}_t$ the exceptional divisor of $\phi_t$. We degenerate the divisors $\mtc{L}_t$ and $\mts{Q}_t$ on general fibers to Weil divisors $\mtc{L}_0$ and $\mts{Q}_0$  on the central fiber, and examine the associated linear series.
However, we encounter a main obstacle that distinguish our case from previous examples (cf. \cite{LX19,SS17}), that is,  the degeneration divisors might not be $\mathbb{Q}$-Cartier. To address this, we perform a $\mathbb{Q}$-factorialization followed by running a minimal model program to obtain a small modification $\wt{\mts{X}}\to \mts{X}$ maintaining the general fibers unchanged, such that both $\wt{\mtc{L}}$ and $\wt{\mts{Q}}$ are $\bQ$-Cartier, and the latter is  ample over $\mts{X}$. While the central fiber of the new family $\wt{\mts{X}}$ is only weak Fano, we achieve the desired $\mathbb{Q}$-Cartier properties.  
Next, we show that $\wt{\mtc{L}}$ is a big and nef Cartier divisor over $T$. This is based on a delicate study of  the moduli spaces of certain lattice-polarized K3 surfaces. In fact, the above construction shows that $-K_{\wt{\mts{X}}} + \epsilon \wt{\mts{Q}}\sim_{\bR, T} 2\wt{\mtc{L}} - (1-\epsilon) \mts{Q}$ is ample over $T$ for $0<\epsilon\ll 1$. By restricting to a general K3 surface $\wt{S}\in |-K_{\wt{\mts{X}}_0}|$ and changing polarization for such lattice-polarized K3 surfaces $(\wt{S}, \Lambda)$ where $\Lambda$ is generated by the restrictions of $\wt{\mtc{L}}$ and $\wt{\mts{Q}}$ (cf. Proposition \ref{isomo} and Lemma \ref{samefamily}), we show that ampleness of $(2\wt{\mtc{L}} - (1-\epsilon) \mts{Q})|_{\wt{S}}$ implies nefness of $\wt{\mtc{L}}|_{\wt{S}}$, which then implies nefness of $\wt{\mtc{L}}$ by lifting sections from the K3 surface to the weak Fano threefold. Thus by the base-point-free theorem, we have a birational morphism $\wt{\mts{X}}\to \mts{V}$ over $T$ by taking the ample model of $\wt{\mtc{L}}$. Then we show that $\wt{\mtc{L}}_0$ is the pull-back of a Cartier divisor $L_V$ on $V=\mts{V}_0$ such that $-K_V \sim 2L_V$, which implies that $V$ is also a cubic threefold by \cite{Fuj90}. Finally, we show that both families $\mts{X}$ and $\wt{\mts{X}}$ coincide with the blow-up of $\mts{V}$ along a section corresponding to a double point in each fiber. 

In Section \ref{VGITHK}, we construct a universal family of blow-ups of $\mathbb{P}^3$ over the base $U$, which parametrizes all $(2,3)$-complete intersection curves, and compute the class of the CM line bundle over $U$. By identifying the CM line bundle with a chosen polarization for GIT quotient, we establish isomorphisms between the K-moduli space $\overline{M}^K_{\text{№2.15}}$ and the VGIT moduli space $\ove{M}^{\GIT}(\frac{2}{5},\frac{1}{2})$, as well as the Hassett-Keel model $\overline{M}_4(\frac{1}{2},\frac{23}{44})$. However, since $U$ is not compact, we may encounter non-complete intersection subschemes when taking the limit along an orbit of some one-parameter subgroup. To address this issue, we introduce a stratification of the VGIT moduli space and discuss on each stratum.

\subsection*{Acknowledgments}

We would like to thank Kenneth Ascher, Dori Bejleri, Harold Blum, Izzet Co\c{s}kun, Kristin DeVleming, Tiago Duarte Guerreiro, Lawrence Ein, Philip Engel, Luca Giovenzana,  Lena Ji, Ananth Shankar, Nivedita Viswanathan, Xiaowei Wang, and Zhiwei Zheng for helpful discussions. We also thank the anonymous referee for the revision suggestions. Research of YL was supported in part by the NSF CAREER Grant DMS-2237139 and the Alfred P. Sloan Foundation.

\section{Preliminaries}

We work over the field $\bC$ of complex numbers.

\subsection{K3 surfaces} We begin by reviewing some basics in the geometry of (quasi-)polarized K3 surfaces, which arise naturally in the study of Fano threefolds.

\begin{defn}
  A \emph{K3 surface} is a normal projective surface $S$ with at worst ADE singularities satisfying $\omega_S\simeq \mathcal{O}_S$ and $H^1(S, \mathcal{O}_S) =0$. A polarization (resp. quasi-polarization) on a K3 surface $S$ is an ample (resp. big and nef) line bundle $L$ on $S$. We call the pair $(S,L)$ a \emph{polarized} (resp. \emph{quasi-polarized}) \emph{K3 surface of degree $d$}, where $d=(L^2)$. Since $d$ is always an even integer, we sometimes write $d = 2k$.
\end{defn}

Let $(S,L)$ be a polarized K3 surface. Then there are three cases based on the behavior of the linear system $|L|$.

\begin{theorem}[cf. \cite{May72, SD}]\label{Mayer}
Let $(S,L)$ be a polarized K3 surface of degree $2k$. Then one of the followings holds.
\begin{enumerate}
    \item \textup{(Generic case)} The linear series $|L|$ is very ample, and the embedding $\phi_{|L|}:S\hookrightarrow |L|^{\vee}$ realizes $S$ as a degree $2k$ surface in $\mb{P}^{k+1}$. In this case, a general member of $|L|$ is a smooth non-hyperelliptic curve.
    \item \textup{(Hyperelliptic case)} The linear series $|L|$ is base-point-free, and the induced morphism $\phi_{|L|}$ realizes $S$ as a double cover of a normal surface of degree $k$ in $\mb{P}^{k+1}$. In this case, a general member of $|L|$ is a smooth hyperelliptic curve, and $|2L|$ is very ample.
    \item \textup{(Unigonal case)}  The linear series $|L|$ has a base component $E$, which is a smooth rational curve. The linear series $|L-E|$ defines a morphism $S\rightarrow \mb{P}^{k+1}$ whose image is a rational normal curve in $\mb{P}^{k+1}$. In this case, a general member of $|L-E|$ is a union of disjoint elliptic curves, and $|2L|$ is base-point-free.
\end{enumerate}
\end{theorem}

\subsection{Geometry of Fano threefolds}\label{sec:2.15}

In this section, let us briefly review the geometry of a smooth member in the deformation family №2.15 of Fano threefolds.

\begin{lemma}
    Any non-hyperelliptic smooth canonical curve of genus four is a smooth complete intersection of a unique quadric surface, which is normal, and a cubic surface in $\mb{P}^3$. The converse also holds true.
\end{lemma}

Let $C\subseteq \mb{P}^3$ be the smooth complete intersection of a normal quadric surface $Q$ and a cubic surface $W$. Consider the linear series $|\mtc{I}_{C}(4)|$, a general member $W$ of which is a quartic K3 surface containing $C$. Then the residue curve of $C$ with respect to $W$ and the quadric $Q$ is a conic curve $\Gamma$. Let $\pi: X\rightarrow \mb{P}^3$ the blow-up of $\mb{P}^3$ along $C$ with exceptional divisor $E$, and $\wt{Q}$ be the strict transform of $Q$. The linear system of cubic surfaces vanishing along C induces a rational map $\mb{P}^3\dashrightarrow \mb{P}^4$, which is resolved by the blow-up $\pi$: $$\xymatrix{
& & X \ar[dl]_{\pi} \ar[dr]^{\phi} &\\
 & \mb{P}^3 \ar@{-->}[rr]^{|I_C(3)|}  &  &  V\subseteq \mb{P}^4.
 }$$ The morphism $\phi$ contracts $\wt{Q}$, whose image is an $A_1$-singularity of $V$, and is an isomorphism elsewhere, and $V$ is a singular cubic hypersurface. Let $H=\pi^{*}\mtc{O}_{\mb{P}^3}(1)$ and $L=\phi^{*}(\mtc{O}_{\mb{P}^4}(1)|_V)$ be two big and nef divisors. Then we have the linear equivalence relations $$-K_X\sim 4H-E,\quad L\sim 3H-E,\quad \wt{Q}\sim 2H-E.$$
 Let $S\in|-K_X|$ be a smooth member. Then $\wt{C}:=S\cap E$ is mapped by $\pi$ to $C$ isomorphically, and $\Gamma:=S\cap \wt{Q}$ is contracted by $\phi$ to an $A_1$-singularity of $\ove{S}:=\phi(S)$. We have that $(S,-K_X|_S)$ is a polarized K3 surface of degree $22$, and $(S,L|_S)$ is a quasi-polarized degree $6$ K3 surface, whose ample model is $(\ove{S},\mtc{O}_{\mb{P}^4}(1)|_{\ove{S}})$.

\begin{lemma}\label{isomofL}
    Let $S\in|-K_X|$ be a general member. Then
    \begin{enumerate}
        \item we have a natural isomorphism $$H^0(X,L)\stackrel{\simeq}{\longrightarrow} H^0(S,L|_S)$$ of $5$-dimensional vector spaces; and
        \item the group $\NS(X)$ is free of rank $2$, generated by $[H|_S]$ and $[E|_S]$, and contains a polarization $(4H-E)|_S$ of degree $22$.
    \end{enumerate}
\end{lemma}

\begin{proof}
    Notice that there exists an exact sequence $$0\longrightarrow \mtc{O}_X(L-S)\longrightarrow \mtc{O}_X(L)\longrightarrow \mtc{O}_S(L)\longrightarrow 0,$$ where $L-S\sim -H$ is anti-effective. By Kawamata-Viehweg vanishing theorem, we have that $$h^0(X,\mtc{O}_X(-H))=h^1(X,\mtc{O}_X(-H))=0.$$ Then the desired isomprhism follows immediately. The dimension of the vector spaces is $\frac{1}{2}(L|_S)^2+2=5$. To prove the second statement, it suffices to observe that a general element $S\in|-K_X|$ is identified with a general quartic K3 surface in $\mb{P}^3$ containing a smooth conic curve $\Gamma$.
    
\end{proof}

\subsection{K-stability and K-moduli spaces}

\begin{defn}
    A \emph{$\mb{Q}$-Fano variety} (resp. \emph{weak $\mb{Q}$-Fano variety}) is a  normal projective variety $X$ such that the anti-canonical divisor $-K_X$ is an ample (resp. a big and nef) $\mb{Q}$-Cartier divisor, and $X$ has klt singularities.

    By \cite{BCHM}, for a weak $\bQ$-Fano variety $X$, the anti-canonical divisor $-K_X$ is always big and semiample whose ample model gives a $\bQ$-Fano variety $\overline{X} :=\Proj R(-K_X)$. We call $\overline{X}$ the \emph{anti-canonical model} of $X$.
\end{defn}

\begin{defn}
    A $\mb{Q}$-Fano variety $X$ (resp. weak $\bQ$-Fano variety) is called \emph{$\mb{Q}$-Gorenstein smoothable} if there exists a projective flat morphism $\pi:\mts{X}\rightarrow T$ over a pointed smooth curve $(0\in T)$ such that the following conditions hold:
    \begin{itemize}
        \item $-K_{\mts{X}/T}$ is $\mb{Q}$-Cartier and $\pi$-ample (resp. $\pi$-big and $\pi$-nef);
        \item $\pi$ is a smooth morphism over $T^\circ:=T\setminus \{0\}$; and
        \item $\mts{X}_0\simeq X$.
    \end{itemize}
\end{defn}

\begin{defn}
Let $X$ be an $n$-dimensional $\mb{Q}$-Fano variety, and $E$ a prime divisor on a normal projective variety $Y$, where $\pi:Y\rightarrow X$ is a birational morphism. Then the \emph{log discrepancy} of $X$ with respect to $E$ is $$A_{X}(E):=1+\coeff_{E}(K_Y-\pi^{*}K_X).$$ We define the \emph{S-invariant} of $X$ with respect to $E$ to be $$S_{X}(E):=\frac{1}{(-K_X)^n}\int_{0}^{\infty}\vol_Y(-\pi^{*}K_X-tE)dt,$$ and the \emph{$\beta$-invariant} of $X$ with respect to $E$ to be $$\beta_{X}(E):=A_{X}(E)-S_{X}(E).$$
\end{defn}

\begin{theodef} \textup{(cf. \cite{Fuj19,Li17,BX19, LWX21})} A $\mb{Q}$-Fano variety $X$ is 
\begin{enumerate}
    \item K-semistable if and only if $\beta_{X}(E)\geq 0$ for any prime divisor $E$ over $X$;
    \item K-stable if and only if $\beta_{X}(E)>0$ for any prime divisor $E$ over $X$;
    \item K-polystable if and only if it is K-semistable and any $\mb{G}_m$-equivariant K-semistable degeneration of $X$ is isomorphic to itself.
\end{enumerate}
A weak $\mb{Q}$-Fano variety $X$ is \textup{K-(semi/poly)stable} if its anti-canonical model $\ove{X}:=\Proj R(-K_X)$ is K-(semi/poly)stable.

\end{theodef}

\begin{defn}
Let $x\in X$ be an $n$-dimensional klt singularity. Let $\pi:Y\rightarrow X$ be a birational morphism such that $E\subseteq Y$ is an exceptional divisor whose center on $X$ is $\{x\}$. Then the \emph{volume} of $(x\in X)$ with respect to $E$ is $$\vol_{x,X}(E):=\lim_{m\to \infty}\frac{\dim\mtc{O}_{X,x}/\{f:\ord_E(f)\geq m\}}{m^n/n!},$$ and the \emph{normalized volume} of $(x\in X)$ with respect to $E$ is $$\wh{\vol}_{x,X}(E):=A_{X}(E)^n\cdot\vol_{x,X}(E).$$ We define the \emph{local volume} of $x\in X$ to be $$\wh{\vol}(x,X):=\ \inf_{E}\ \wh{\vol}_{x,X}(E),$$ where $E$ runs through all the prime divisors over $X$ whose center on $X$ is $\{x\}$.
\end{defn}

\begin{theorem}[cf. \cite{Liu18}]\label{volume}
Let $X$ be an $n$-dimensional K-semistable (weak) $\mb{Q}$-Fano variety, and $x\in X$ be a point. Then we have the inequality $$(-K_X)^n \ \leq\  \left(\frac{n+1}{n}\right)^n\cdot \wh{\vol}(x,X).$$
\end{theorem}

\begin{proof}
If $X$ is a $\bQ$-Fano variety, the result follows from \cite{Liu18}. If $X$ is a weak $\bQ$-Fano variety, let $\phi: X\to \overline{X}$ be its anti-canonical model. Then we have $\phi^*(-K_{\overline{X}}) = -K_X$ and hence 
\[
(-K_X)^n \ =\  (-K_{\overline{X}})^n\ \leq\  \left(\frac{n+1}{n}\right)^n\cdot \wh{\vol}(\phi(x),\overline{X})\ \leq\  \wh{\vol}(x,X),
\]
where the last inequality follows from \cite[Lemma 2.9(2)]{LX19}.

\end{proof}


Now we introduce the CM line bundle of a flat family of $\mb{Q}$-Fano varieties, which is a functorial line bundle over the base (cf. \cite{PT06,PT09,Tia97}). 

Let $\pi:\mts{X}\rightarrow S$ be a proper flat morphism of connected schemes with $S_2$ fibers of pure dimension $n$, and $\mtc{L}$ be an $\pi$-ample line bundle on $\mts{X}$. By \cite{KM76}, there are line bundles $\lambda_i=\lambda_i(\mts{X},\mtc{L})$ on $S$ such that $$\det(\pi_{!}(\mtc{L}^k))=\lambda_{n+1}^{\otimes\binom{k}{n+1}}\otimes \lambda_n^{\otimes\binom{k}{n}}\otimes \cdots\otimes\lambda_0^{\otimes\binom{k}{0}}$$ for any $k\gg0$. Write the Hilbert polynomial for each fiber $\mts{X}_s$ as $$\chi(\mts{X}_s,\mtc{L}^k_{s})=b_0k^n+b_1k^{n-1}+O(k^{n-2}).$$
 
\begin{defn}
    The \emph{CM $\mb{Q}$-line bundle} of the polarized family $(\pi:\mts{X}\rightarrow S,\mtc{L})$ are $$\lambda_{\CM,\pi,\mtc{L}}:=\lambda_{n+1}^{n(n+1)+\frac{2b_1}{b_0}}\otimes \lambda_n^{-2(n+1)}.$$
    If, in addition, both $\mts{X}$ and $S$ are normal, and $-K_{\mts{X}/S}$ is a $\pi$-ample $\bQ$-Cartier divisor, then we write
    $\lambda_{\CM,\pi} = l^{-n}\lambda_{\CM, \pi, \mtc{L}}$ where  $\mtc{L} = -l K_{\mts{X}/S}$ is a $\pi$-ample line bundle for some $l\in \bZ_{>0}$.
\end{defn}

\begin{theorem}\label{KimpliesGIT} \textup{(cf. \cite[Thm.2.22]{ADL19})}
Let $f:\mts{X}\rightarrow S$ be a $\mb{Q}$-Gorenstein family of $\bQ$-Fano varieties over a normal projective base $S$. Let $G$ be a reductive group acting on $\mts{X}$ and $S$ such that $f$ is $G$-equivariant. Moreover, assume the following conditions are satisfied:
\begin{enumerate}[(i)]
\item for any $s\in S$, if $\Aut(\mts{X}_s)$ is finite, then the stabilizer $G_s$ is also finite;
\item if we have $\mts{X}_s\simeq \mts{X}_{s'}$ for $s,s'\in S$, then $s'\in G\cdot s$;
\item $\lambda_{\CM,f}$ is an ample $\mb{Q}$-line bundle on $S$.
\end{enumerate}
Then $s\in S$ is a GIT-(poly/semi)stable point with respect to the $G$-linearized $\mb{Q}$-line bundle $\lambda_{\CM,f}$ if $\mts{X}_s$ is a K-(poly/semi)stable $\mb{Q}$-Fano variety.
\end{theorem}

The following theorem is usually called the \emph{K-moduli Theorem}, which is attributed to many people (cf. \cite{ABHLX20,BHLLX21,BLX19,BX19,CP21,Jia20,LWX21,LXZ22,Xu20,XZ20,XZ21}).

\begin{theorem}[K-moduli Theorem]\label{kmoduli}
Fix two numerical invariants $n\in \bN$ and $V\in \bQ_{>0}$. Consider the moduli functor $\mts{M}^K_{n,V}$ sending a  base scheme $S$ to

\[
\left\{\mts{X}/S\left| \begin{array}{l} \mts{X}\to S\textrm{ is a proper flat morphism, each geometric fiber}\\ \textrm{$\mts{X}_{\bar{s}}$ is an $n$-dimensional K-semistable $\bQ$-Fano variety of}\\ \textrm{volume $V$, and $\mts{X}\to S$ satisfies Koll\'ar's condition}\end{array}\right.\right\}.
\]
Then there is an Artin stack, still denoted by $\mts{M}^K_{n,V}$, of finite type over $\mb{C}$ with affine diagonal which represents the moduli functor. The $\mb{C}$-points of $\mts{M}^K_{n,V}$ parameterize K-semistable  $\mb{Q}$-Fano varieties $X$ of dimension $n$ and volume $V$. Moreover, the Artin stack $\mts{M}^K_{n,V}$ admits a good moduli space $\ove{M}^K_{n,V}$, which is a projective scheme, whose $\mb{C}$-points parameterize K-polystable $\mb{Q}$-Fano varieties. The CM $\mb{Q}$-line bundle $\lambda_{\CM}$ on $\mts{M}^K_{n,V}$ descends to an ample $\mb{Q}$-line bundle $\Lambda_{\CM}$ on $\ove{M}^K_{n,V}$.
\end{theorem}

\begin{theorem}[cf. {\cite[Proposition 4.33]{CA21}}]\label{thm:general-Kstable}
A general smooth member of the family \textnumero 2.15 of Fano threefolds is K-stable.
\end{theorem}

\begin{defn}
Let $\mts{M}^K_{\textup{№2.15}}$ (resp. $\ove{M}^K_{\textup{№2.15}}$) be the irreducible component of $\mts{M}^K_{3,22}$ (resp. $\ove{M}^K_{3,22}$) with reduced stack (resp. scheme) structure whose general point parameterizes a K-stable blow-up of $\mb{P}^3$ along a smooth $(2,3)$-complete intersection curve, which is necessarily non-empty by Theorem \ref{thm:general-Kstable}. We call $\mts{M}^K_{\textup{№2.15}}$ (resp. $\ove{M}^K_{\textup{№2.15}}$) the K-moduli stack (resp. the K-moduli space) of the family \textnumero 2.15 of Fano threefolds.
\end{defn}

\begin{remark}\textup{
   By construction, a K-moduli stack $\mts{M}^K_{n,V}$ (resp. a K-moduli space $\ove{M}^K_{n,V}$) could be non-reduced, reducible or disconnected. See \cite{KP21, Pet21, Pet22}  for examples along this direction. In practice, an explicit K-moduli spaces we consider is often an irreducible component of $\ove{M}^K_{n,V}$ for some $n,V$, whose general point parametrizes a K-polystable (usually smooth) Fano variety  that we are interested in. As we shall see later in Corollary \ref{cor:stack-smooth}, the connected component of $\mts{M}^K_{3,22}$ containing the locus of K-stable smooth Fano threefolds of family \textnumero 2.15 is indeed a smooth irreducible stack. Thus a posteriori we could also define $\mts{M}^K_{\textup{№2.15}}$ as this connected component since smoothness holds.}
\end{remark}

\subsection{Variation of GIT for $(2,3)$-complete intersection curves}\label{VGIT}

In this part, we briefly introduce the setup in \cite{CMJL14} of Variation of GIT (abbv. VGIT) for $(2,3)$-complete intersection curves.

Let $\mb{P}^9=\mb{P}H^0(\mb{P}^3,\mtc{O}_{\mb{P}^3}(2))$ be the parameter space of quadric surfaces in $\mb{P}^3$, and $\cQ\subseteq \mb{P}^3\times \mb{P}^9$ the universal family of quadric surfaces. Let $\pi_1:\mb{P}^3\times\mb{P}^9\rightarrow\mb{P}^3$ and $\pi_2:\mb{P}^3\times\mb{P}^9\rightarrow\mb{P}^9$ be the two projections. Consider the locally free sheaf $\mtc{E}:=\pi_{2*}(\mtc{O}_{\cQ}\otimes\pi_{1}^{*}\mtc{O}_{\mb{P}^3}(3))$ of rank 16 on $\mb{P}^9$ and set $\pi:\mb{P}\mtc{E}\rightarrow\mb{P}^9$ to be the corresponding projective bundle. Then one has $$\Pic(\mb{P}\mtc{E})\ \simeq\  \mb{Z}\cdot \eta\ \oplus\  \mb{Z}\cdot\xi,$$ where $\eta=\pi^{*}\mtc{O}_{\mb{P}^9}(1)$ and $\xi=\mtc{O}_{\mb{P}\mtc{E}}(1)$. We denote the class of a hyperplane in $\mb{P}^3$ by $h$.

We first compute the Chern character of $\mtc{E}$. Notice that $\mtc{E}$ fits into a short exact sequence \begin{equation}\label{*}
    0\longrightarrow \pi_{2*}(\mtc{I}_{\cQ}\otimes\pi_{1}^{*}\mtc{O}_{\mb{P}^3}(3))\longrightarrow \pi_{2*}(\pi_{1}^{*}\mtc{O}_{\mb{P}^3}(3))\longrightarrow \mtc{E}=\pi_{2*}(\mtc{O}_{\cQ}\otimes\pi_{1}^{*}\mtc{O}_{\mb{P}^3}(3))\longrightarrow 0,
\end{equation} and that 
\begin{align*}
    \pi_{2*}(\mtc{I}_{\cQ}\otimes\pi_{1}^{*}\mtc{O}_{\mb{P}^3}(3)) & =H^0(\mb{P}^3,\mtc{O}_{\mb{P}^3}(1))\otimes \mtc{O}_{\mb{P}^9}(-1),\\\pi_{2*}(\pi_{1}^{*}\mtc{O}_{\mb{P}^3}(3))& =H^0(\mb{P}^3,\mtc{O}_{\mb{P}^3}(3))\otimes \mtc{O}_{\mb{P}^9}.
\end{align*}
It follows that $$\che(\mtc{E})=20-4\cdot\che(\mtc{O}_{\mb{P}^9}(-1))=20-4\cdot\sum_{i=0}^{\infty}\frac{(-1)^i\eta^i}{i!}.$$ In particular, we have $c_1(\mtc{E})=4\eta$ and hence $$\omega_{\mb{P}\mtc{E}}=\pi^{*}(\omega_{\mb{P}^9}\otimes \det(\mtc{E}^{*}))\otimes \mtc{O}_{\mb{P}\mtc{E}}(-\rk\mtc{E})=-14\eta-16\xi.$$

Consider the natural $\PGL(4)$-action on $\mb{P}\mtc{E}$ induced from the action on $\mb{P}^3$. For any rational number $t>0$, we take a $\mb{Q}$-line bundle of the class $\eta+t\xi$. It follows from \cite[Theorem 2.7]{Ben14} that the line bundle is ample if and only if $t<\frac{1}{2}$. 

\begin{defn}\label{gitdefn}
    For any $t\in(0,\frac{1}{2})$, we define the quotient stack $$\mts{M}^{\GIT}(t)\ := \ \left[\mb{P}\mtc{E}^{\sst}(t)/ \PGL(4)\right],$$ which is an Artin stack. Moreover, $\mts{M}^{\GIT}(t)$ admits a good moduli space $$\ove{M}^{\GIT}(t)\ := \ \mb{P}\mtc{E}\sslash_t \PGL(4),$$ which is the (projective) VGIT moduli space parametrizing $(2,3)$-complete intersection curves.
\end{defn}

\begin{remark}\label{vgitremark}\textup{Here we collect results from \cite{CMJL14} on the VGIT moduli spaces.
  \begin{enumerate}
      \item Although $\eta+t\xi$ is not ample when $t\geq \frac{1}{2}$, one can still define the VGIT moduli spaces $\ove{M}^{\GIT}(t)$ for $\frac{1}{2}\leq t\leq \frac{2}{3}$ (cf. \cite[Section 6]{CMJL14}).
      \item There are five VGIT walls in total $$T_0=0, \ \  T_1=\frac{2}{9},\ \ T_2=\frac{2}{5},\ \  T_3=\frac{1}{2},\ \ T_4=\frac{2}{3}.$$  By structure theorem of VGIT, if $t\in (T_i,T_{i+1})$, then $\ove{M}^{\GIT}(t)$ is independent of the choice of $t$. Therefore, we write $\ove{M}^{\GIT}(T_i,T_{i+1})$ to denote $\ove{M}^{\GIT}(t)$ for any $t\in (T_i,T_{i+1})$.
      \item The GIT quotient space $\ove{M}^{\GIT}(\frac{2}{3})$ is isomorphic to the moduli space of Chow-polystable curves of genus four and degree $6$ in $\bP^3$; while the moduli $\ove{M}^{\GIT}(0,\frac{2}{9})$ is isomorphic to the GIT quotient $|\mtc{O}_{\mb{P}^1\times\mb{P}^1}(3,3)|\sslash \SO(4)$.
  \end{enumerate}  
}\end{remark}

\subsection{Moduli of curves and Hassett--Keel program}\label{HKP}

\begin{theorem}
    Let $g\geq 2$ be an integer, and $\ove{\mts{M}}_g$ be the  moduli stack of Deligne-Mumford stable (abbv. DM-stable) curves of genus $g$. The following properties hold.
\begin{enumerate}
    \item  $\ove{\mts{M}}_g$ is an irreducible smooth proper Deligne-Mumford stack. It admits a coarse moduli space $\ove{M}_g$ as a projective variety of dimension $3g-3$ with quotient singularities.
    \item The rational Picard group $\Pic_{\mb{Q}}(\ove{\mts{M}}_g)$ is generated by the Hodge line bundle $\lambda_{\Hodge}$, the class of closure of the locus parametrizing irreducible singular curves $\delta_0$, and $\delta_i$, and the class of closure of the locus parametrizing reducible singular curves with a component of genus $i$, where $1\leq i\leq \lfloor g/2\rfloor$.
    \item the divisor class $K_{\ove{\mts{M}}_g}+\delta$ descends to an ample divisor on $\ove{M}_g$, where $\delta=\sum_{i=0}^{\lfloor g/2\rfloor} \delta_i$.
\end{enumerate}
\end{theorem}

The Hassett--Keel program was initiated by \cite{Has05}, aiming to give modular interpretations of certain log canonical models of the moduli space $\ove{M}_g$ of stable curves of fixed genus $g$. For any rational number $\alpha\in[0,1]$ such that $K_{\ove{\mts{M}}_g}+\alpha \delta$ is pseudo-effective, we can define $$\ove{M}_g(\alpha):=\Proj \bigoplus_{n\geq0}H^0(\ove{\mts{M}}_g, \lfloor n(K_{\ove{\mts{M}}_g}+\alpha \delta)\rfloor).$$ For large genus, the first three models are understood (cf. \cite{HH09,HH13,AFS17,AFS172,AFS173}), but completing the program entirely still seems out of reach. On the other hand, some cases of low genus curves were more or less clear in the past decade (cf. \cite{HL10,CMJL14,Fed12,Fed13,Mul14,zha23b}).

The authors to \cite{CMJL14} studied Hassett-Keel program for genus four curves via the following essential observation. 

\begin{theorem}\textup{(cf. \cite[Theorem 7.1]{CMJL14})}
    Let $\ove{M}^{\GIT}(t)$ be the GIT moduli space of $(2,3)$-complete intersection curves (cf. Definition \ref{gitdefn}). Then each of the log minimal models $\ove{M}_4(\alpha)$ for $\frac{8}{17}\leq \alpha\leq \frac{5}{9}$ is isomorphic to one of the GIT quotients $\ove{M}^{\GIT}(t)$ via the relation $t=\frac{34\alpha-16}{33\alpha-14}$.
\end{theorem}

The next theorem describes part of the Hassett--Keel program for genus four curves. 

\begin{theorem}\label{HassettK}\textup{(cf. \cite[Main Theorem]{CMJL14})}
Let $0<\alpha \leq \frac{5}{9}$ be a rational number. Then the log minimal models $\overline{M}_4(\alpha)$ arise as GIT quotients of the parameter space $\bP E$.  Moreover, the VGIT problem gives us the following diagram:

\begin{equation}
\xymatrix @R=.07in @C=.07in{
 & \overline{M}_4 ( \frac{5}{9} , \frac{23}{44} ) \ar@{-->}[rr]\ar[ldd] \ar[rdd] & & \overline{M}_4 ( \frac{23}{44} , \frac{1}{2} ) \ar[ldd] \ar[rdd] \ar@{-->}[rr]& & \overline{M}_4 ( \frac{1}{2} , \frac{29}{60} ) \ar[ldd] \ar[rdd] & \\
&&&&&& \\
\overline{M}_4 ( \frac{5}{9} ) & & \overline{M}_4 ( \frac{23}{44} ) & & \overline{M}_4 ( \frac{1}{2} ) & & \overline{M}_4 [\frac{29}{60},\frac{8}{17}) \ar[dd] \\
&&&&&&\\
 & & & & & & \overline{M}_4 (\frac{8}{17} )   = \{*\}
 }
 \end{equation}

 \begin{itemize}
 \item[i)] the end point $\overline{M}_4 ( \frac{8}{17}+\epsilon)$ for $0<\epsilon\ll 1$ is obtained via GIT for $(3,3)$ curves on $\bP^1 \times \bP^1$;
 \item[ii)] the other end point $\overline{M}_4 ( \frac{5}{9} )$ is obtained via GIT for the Chow variety of genus 4 canonical curves;
 \item[iii)] the map to $\overline{M}_4 ( \frac{29}{60} )$ contracts the Gieseker--Petri divisor to a point, corresponding to a triple conic;
 \item[iv)] the flip at $\alpha=\frac{1}{2}$ removes the locus of nodal curves whose normalization is hyperelliptic, replacing them with the union of a conic and a double conic in a quadric surface.
 \end{itemize}
\end{theorem}

\section{Moduli spaces of K3 surfaces}\label{K3SURFACES}

In this section, we prove some results on the geometry of moduli spaces of (lattice-) polarized K3 surfaces. This is useful in understanding the anti-canonical linear series of a (weak) Fano variety.

\subsection{Geometry and moduli of K3 surfaces}

\begin{defn}
    Let $L_{K3}:=\mb{U}^{\oplus3}\oplus \mb{E}_8^{\oplus2}$ be a fixed (unique) even unimodular lattice of signature $(3,19)$.
\end{defn}

Let $\Lambda$ be a rank $r$ primitive sublattice of $L_{K3}$ with signature $(1,r-1)$. 
A vector $h\in \Lambda\otimes \mb{R}$ is called \emph{very irrational} if $h\notin \Lambda'\otimes \mb{R}$ for any primitive proper sublattice $\Lambda'\subsetneq \Lambda$. Fix a very irrational vector $h$ with $(h^2)>0$.

\begin{defn}
    A \emph{$\Lambda$-polarized K3 surface} (resp. a \emph{$\Lambda$-quasi-polarized K3 surface}) $(X,j)$ is a K3 surface $X$ with ADE singularities (resp. a smooth projective surface $X$) together with a primitive lattice embedding $j:\Lambda\hookrightarrow\Pic(X)$ such that $j(h)\in \Pic(X)_{\mb{R}}$ is ample (resp. big and nef).
    \begin{enumerate}
        \item Two such pairs $(X_1,j_1)$ and $(X_2,j_2)$ are called \emph{isomorphic} if there is an isomorphism $f:X_1\stackrel{\sim}{\rightarrow} X_2$ of K3 surfaces such that $j_1=f^{*}\circ j_2$.
        \item The \emph{$\Lambda$-(quasi-) polarized period domain} is $$\mb{D}_{\Lambda}:=\mb{P}\{w\in {\Lambda}^{\perp}\otimes\mb{C}:(w^2)=0,\ (w.\ove{w})>0\}.$$
    \end{enumerate}
    When $r=1$, i.e. $\Lambda$ is of rank one, it is convenient to choose $h$ to be the effective generator $L$ of $\Lambda$. We denote by $d$ the self-intersection of $L$, and we call $(X,L)$ a (quasi-)polarized K3 surface of degree $d$.
\end{defn}



\begin{defn}
   For a fixed lattice $\Lambda$ with a very irrational vector $h$, one define the \emph{moduli functor $\mts{F}_{\Lambda}$ of $\Lambda$-polarized K3 surfaces} to send a base scheme $T$ to 
\[
\left\{(f:\mts{X}\rightarrow T;\varphi)\left| \begin{array}{l} \mts{X}\to T\textrm{ is a proper flat morphism, each geometric fiber}\\ \textrm{$\mts{X}_{\bar{t}}$ is an ADE K3 surface, and $\varphi:\Lambda\longrightarrow\Pic_{\mts{X}/T}(T)$ is }\\ \textrm{a group homomorphism such that the induced map }\\ \textrm{$\varphi_{\bar{t}}:\Lambda\rightarrow \Pic(\mts{X}_{\bar{t}})$ is an isometric primitive embedding of}\\ \textrm{lattices and that $\varphi_{\bar{t}}(h)\in \Pic(\mts{X}_{\bar{t}})_{\mb{R}}$ is an ample class.} \end{array}\right.\right\}.
\]

\end{defn}


\begin{theorem}[cf. \cite{Dol96, AE23}]\label{isommoduli}
    The moduli functor of $\Lambda$-polarized K3 surfaces is represented by a smooth separated Deligne-Mumford (DM) stack $\mts{F}_{\Lambda}$. 
    Moreover, $\mts{F}_{\Lambda}$ admits a normal coarse moduli space $F_{\Lambda}$, whose analytification is isomorphic to $\mb{D}_{\Lambda}/\Gamma$, where $\Gamma:=\{\gamma\in \mathrm{O}(L_{K3}):\gamma|_{\Lambda}=\Id_{\Lambda} \}$.
\end{theorem}

\begin{remark}
    \textup{When $\Lambda$ is of rank one, we denote by $d$ the self-intersection of a generator of $\Lambda$, and by $\mts{F}_d$ (resp. $F_d$) the corresponding moduli stack (resp. coarse moduli space).}
\end{remark}

\subsection{K3 surfaces in anticanonical linear series}\label{k33}

Let $\mathscr{H}_{6,\nodal}\subseteq \mts{F}_6$ be the divisor in the moduli stack of polarized K3 surfaces of degree $6$ parametrizing singular K3 surfaces, and $\mathscr{H}_{6,\nodal}^{\nu}\rightarrow \mathscr{H}_{6,\nodal}$ be the normalization. 
Let $\Lambda_0$ be a rank $2$ hyperbolic sublattice with generator $L,Q$ satisfying that $$(L^2)=6,\quad (L.Q)=0,\quad \textup{and}\quad (Q^2)=-2,$$ and signature $(1,1)$. 
Let $J:\Lambda_0\hookrightarrow L_{K3}$ be a primitive embedding. Let $\mts{F}_{\Lambda_1}$ (resp. $\mts{F}_{\Lambda_2}$) be the moduli space of $\Lambda_0$-polarized K3 surface with respect to the very irrational vector $h=h_1:=L-\epsilon Q$ (resp. $h=h_2:=2L-(1-\epsilon) Q$) where $0<\epsilon \ll 1$ is an irrational number. By \cite[Section 2 of arXiv version 1]{AE23}, for each $i\in \{1,2\}$ the universal K3 surface over the moduli stack $\mts{F}_{\Lambda_i}$ is independent of the choice of $\epsilon$. As we shall see later,  these two universal families are indeed isomorphic by Lemma \ref{samefamily}.

\begin{prop}\label{isomo}
    There is an isomorphism of moduli stacks $$\mathscr{H}_{6,\nodal}^{\nu}\stackrel{\simeq}{\longrightarrow} \mts{F}_{\Lambda_1}.$$ Moreover, the total space of the universal family over $\mts{F}_{\Lambda_1}$ is the blow-up of the universal family over $\mts{H}_{6,\nodal}^{\nu}$ along a section in the singular locus. 
\end{prop}

\begin{proof}
    Let $f:(\mts{S},\mtc{L}-\epsilon \mts{Q})\rightarrow \mts{F}_{\Lambda_1}$ be the universal family of K3 surfaces. Here $\mtc{L}\in \Pic_{\mts{S}/\mts{F}_{\Lambda_1}}(\mts{F}_{\Lambda_1})$ is the descent of a line bundle on the total space after base change to an \'etale cover modulo pull-back of line bundles on the base (see e.g. \cite[Section 5.1]{Huy16}).  Notice that $\mtc{L}$ is $f$-nef and $f$-big as $0<\epsilon \ll 1$. Taking the $f$-ample model with respect to $\mtc{L}$, one obtains a family of polarized K3 surfaces of degree $6$ $$\xymatrix{
 & (\mts{S},\mtc{L}) \ar[rr] \ar[dr]_{{f}}  &  &  (\ove{\mts{S}},\ove{\mtc{L}}) \ar[dl]^{\ove{f}}\\
 & & \mts{F}_{\Lambda_1}  &\\
 }$$ Observe that for a general $\ove{f}$-fiber, there is an $A_1$-singularity (i.e. an ordinary double point). Then by the universal property of $\mts{F}_6$, there exists a natural morphism $\mts{F}_{\Lambda_1}\rightarrow \mts{F}_6$, whose image lies in $\mathscr{H}_{6,\nodal}$. As $\mts{F}_{\Lambda_1}$ is normal by Theorem \ref{isommoduli}, then the morphism factors through a morphism $\phi:\mts{F}_{\Lambda_1}\rightarrow\mathscr{H}_{6,\nodal}^{\nu}$, which descends to a morphism $\varphi: F_{\Lambda_1}\rightarrow H^{\nu}_{6,n}$. We shall show that $\phi$ is a representable finite birational morphism between normal DM stacks, which would imply that it is an isomorphism by Zariski's main theorem.

We start with the quasi-finiteness of the morphism $\phi:\mts{F}_{\Lambda_1}\rightarrow \mts{F}_6$. For every degree $6$ polarized (singular) K3 surface $(\ove{S},\ove{L})$, there are finitely many singularities. Denote by $\beta:\wt{S}\rightarrow\ove{S}$ the minimal resolution of $\ove{S}$, which is still a K3 surface, and $\wt{L}:=\beta^*\ove{L}$. Notice that every preimage $(S,\Lambda)$, where $\Lambda$ is generated by $L$ and $Q$, of $(\ove{S},\ove{L})$ under $\phi$ satisfies that $\beta$ factor as $\wt{S}\rightarrow S\rightarrow \ove{S}$, and there are finitely many such $S$. Moreover, for each such an $S$, since $Q\in L^{\perp}$ is an integral class and by Hodge index theorem $L^{\perp}$ is negative definite, there are finitely many choices for $Q$. This proves the quasi-finiteness. Moreover, $\phi$ is birational as a general K3 surface in $\mts{H}_{6, \nodal}$ has only one $A_1$-singularity.

 To prove that $\phi$ is representable, it suffices to check that the morphism $\mts{F}_{\Lambda_1}\to \mts{F}_6$ induces an injection of stabilizers, which would imply that $\phi$ induces an injection of stabilizers. As any stabilizer $\sigma$ of $(S,L)\in \mts{F}_{\Lambda_1}$ preserves the lattice as well as the very irrational vector, then it sends $L$ to $L$. Therefore, $\sigma$ induces an isomorphism on the ample model $\ove{S}$ of $S$ with respect to $L$.
 
 Now let us prove that $\phi$ is proper, and hence finite. Suppose that $f:\mts{S}\rightarrow (0\in B)$ is a 1-dimensional family of degree $6$ polarized K3 surfaces such that there is a commutative diagram $$\xymatrix{
 & \wt{\mts{S}}^{\circ} \ar[rr]^{g} \ar[dr]_{\wt{f}^{\circ}}  &  &  \mts{S}^{\circ} \ar[dl]^{f^{\circ}}\\
 & & B^{\circ}  &\\
 },$$ where $B^{\circ}=B-\{0\}$, and $g$ is obtained by taking ample model over $B^{\circ}$ with respect to $\mtc{L}$. We may also assume that $g_b:\wt{\mts{S}}_b\rightarrow \mts{S}_b$ contracts a unique conic to an $A_1$-singularity $p_b$ for any $b\in B^{\circ}$. Taking the closure of $\{p_b\}_{b\in B^{\circ}}$ in $\mts{S}$, one obtains a closed subscheme $\mts{P}$. Then $\wt{f}:\Bl_{\mts{P}}\mts{S}\rightarrow B$ coincide with $\wt{f}^{\circ}$ over $B^{\circ}$. Moreover, by Lemma \ref{ampleness}, the family $\wt{f}$ yields a filling of $B^{\circ}\rightarrow \mts{F}_{\Lambda_1}$. If there is another filling, corresponding to a $\Lambda_1$-polarized K3 surface $S$, then by the universal property of blow-up, there is a natural birational morphism $S\rightarrow \mts{\wt{S}}_0$. On the other hand, both $S$ and $\mts{\wt{S}}_0$ are the ample model of $S$ with respect to $L-\epsilon Q$, and hence they are isomorphic. 
 
 Therefore, by Zariski's Main Theorem for DM stacks, we conclude that $\phi$ is an isomorphism.

\end{proof}

\begin{lemma}\label{ampleness}
    Let $(S,L)$ be a polarized K3 surface of degree $6$ with a singularity $p\in S$. Let $\mu:\wt{S}:=\Bl_pS \rightarrow S$ be the blow-up at $p$ with exceptional divisor $Q$. Then $2\mu^{*}L-Q$ is nef. In particular, for any $0<\epsilon\ll1$, the $\mb{R}$-Cartier $\mb{R}$-divisor $2\mu^{*}L-(1-\epsilon)Q$ is ample.
\end{lemma}

\begin{proof}
    Let $C\subseteq \wt{S}$ be an irreducible curve. If $C$ is $\mu$-exceptional, then we have that $(2\mu^{*}L-{Q}.C)>0$ since $2\mu^{*}L-{Q}$ is $\mu$-ample. Now we assume $C$ is not exceptional so that $D:=\mu_{*}C\subseteq S$ is a curve. If $(S,L)$ is not unigonal, then $2L$ is very ample, and hence $$(2\mu^{*}L-{Q}.C)=\deg_{2L}D-\mult_pD\geq0.$$ Suppose that $(S,L)$ is unigonal and $f:S\rightarrow \mb{P}^1$ is a Weierstrass elliptic fibration with a section $E$. Then $L \sim 4F + E$ where $F$ is a fiber of $f$. Since $p$ is a singular point in $S$, then $p$ is also a singular point of the fiber $F_p$ containing $p$. Note that $F_p$ is integral as $f$ is a Weierstrass fibration, see e.g. \cite[Section 4.3]{ADL22}. If $D$ is neither $F_p$ nor the section $E$, then $$(2\mu^{*}L-{Q}.C)\ =\  (8F +2E .D)-\mult_p D\ \geq\  (8F_p.D)-\mult_p D\ \geq\  7 \mult_p D\ \geq\ 0,$$ since $F_p$ is a Cartier divisor containing $p$. 
    If $D= F_p$, then we have that $$(2\mu^{*}L.C)\ =\ 2(4F+E.F_p)\ =\ 2,\  \ \textup{and}\ \ \ ({Q}.C)\ =\ \mult_p F_p\ =\ 2,$$ as $p$ is either cuspidal or nodal in $F_p$. If $D=E$, then $p\not\in E$ which implies 
    \[
    (2\mu^{*}L-{Q}. C) \ =\  (2L.E)\  =\  (8F + 2E.E)\ =\ 4\  >\ 0.
    \]
    As a result, $2\mu^*L - Q$ is nef on $\wt{S}$.

    Notice that for $0<\epsilon\ll1$, the $\mb{R}$-divisor $\mu^{*}L-\epsilon {Q}$ is ample, and so is $2\mu^{*}L-(1-\epsilon){Q}$, which can be written as a positive linear combination of $2\mu^{*}L-{Q}$ and $\mu^{*}L-\epsilon {Q}$.
    
\end{proof}

\begin{lemma}\label{samefamily}
    The moduli stacks $\mts{F}_{\Lambda_1}$ and $\mts{F}_{\Lambda_2}$ are isomorphic. Moreover, their universal K3 surfaces are isomorphic. 
\end{lemma}

\begin{proof}
    By Proposition \ref{isomo}, each fiber over $\mts{F}_{\Lambda_1}$ is the blow-up of a singular polarized degree $6$ K3 surface at a singularity. The it follows from Lemma \ref{ampleness} that the $\mb{R}$-divisor $2L-(1-\epsilon)Q$ is ample. Therefore, by the universality of the moduli stack $\mts{F}_{\Lambda_2}$, there exists a natural morphism $\phi:\mts{F}_{\Lambda_1}\rightarrow\mts{F}_{\Lambda_2}$. Then $\phi$ descends to a morphism between their coarse moduli spaces $\psi:{F}_{\Lambda_1}\rightarrow{F}_{\Lambda_2}$, which is an isomorphism by Theorem \ref{isommoduli}. 
    
    We claim that the morphism $\phi$ is finite. By \cite[Proposition 6.4]{Alp13}, it remains to check that $\phi$ is representable, quasi-finite, separated, and send closed points to closed points. Notice that the separatedness of $\phi$ follows immediately from the separatedness of the moduli stacks of lattice-polarized K3 surfaces. As both $\mts{F}_{\Lambda_1}$ and $\mts{F}_{\Lambda_2}$ are DM stacks, then every point is closed. Hence $\phi$ is bijective as $\psi$ is an isomorphism between coarse moduli spaces, which implies that $\phi$ is quasi-finite. To show that $\phi$ is representable, it suffices to observe that it preserves the stabilizers. Indeed, the surfaces over $\mts{F}_{\Lambda_1}$ and $\mts{F}_{\Lambda_2}$ have a one-to-one correspondence since $\psi$ is an isomorphism, and any two surfaces under this correspondence are isomorphic by Lemma \ref{ampleness} so they have the same automorphism group.

    Since $\mts{F}_{\Lambda_1}$ and $\mts{F}_{\Lambda_2}$ are both smooth stacks, and $\phi$ is a finite birational  morphism, then $\phi$ is an isomorphism by the Zariski's main theorem, and the universal K3 surfaces over $\mts{F}_{\Lambda_1}$ and $\mts{F}_{\Lambda_2}$ are also isomorphic by the construction of $\phi$ and Lemma \ref{ampleness}.
    
\end{proof}

\begin{remark}\textup{The results on moduli of K3 surfaces will be used in Section \ref{limitkss}: to show that the K-semistable limit is the blow-up of $\mb{P}^3$ along a $(2,3)$-complete intersection curve, we in fact show that it is the blow-up of a singular cubic threefold at a double point. When switching from one birational model to the other, to analyze the behavior of the linear series on the limiting threefold, we need to restrict it first to the anticanonical K3 surfaces.
    }
\end{remark}

\begin{remark}\textup{
    By Lemma \ref{samefamily}, we can write $\mts{F}_{\Lambda_0}$ to be $\mts{F}_{\Lambda_1}\simeq \mts{F}_{\Lambda_2}$, and the universal family $u:\mts{S}_{\Lambda_0}\rightarrow \mts{F}_{\Lambda_0}$.}
\end{remark}

\begin{remark}
\textup{Indeed, it is shown in \cite[Section 2 of arXiv version 1]{AE23} that for any fixed lattice $\Lambda_0$ and two  very irrational vectors $h_1, h_2\in \Lambda_0\otimes \bR$, the moduli stacks $\mts{F}_{\Lambda_1}$ and $\mts{F}_{\Lambda_2}$ are isomorphic, where $\Lambda_i$ denotes the lattice $\Lambda_0$ with polarization $h_i$. However, their universal K3 surfaces may not be isomorphic but only birational on each fiber. If, in addition, $h_1$ and $h_2$ belong to the same so-called \emph{small cone} (see \cite[Section 2 of arXiv version 1]{AE23}), then their universal K3 surfaces are isomorphic. In some sense, Lemma \ref{samefamily} precisely shows that $h_1=L-\epsilon E$ and $h_2 = 2L-(1-\epsilon)E$ lie in the same small cone in our case.}
\end{remark}

We end this section with a result on degenerations of K3 surfaces that are anti-canonical divisors of Fano threefolds from family \textnumero 2.15.

\begin{defn}
    Let $S$ be a smooth K3 surface. We call $S$ a \emph{$(4,22)$-polarized K3 surface} if $S$ is isomorphic to a smooth quartic surface in $\mb{P}^3$ which contains a smooth conic curve.
\end{defn}

Let $S\subseteq \mb{P}^3$ be a $(4,22)$-polarized K3 surface. Observe that the Picard group of $S$ is isometric to the lattice $\Lambda_0$. There are two natural polarization associated to $X$: $\mtc{O}_{\mb{P}^3}(1)|_S$ and $\mtc{O}_{\mb{P}^3}(2)|_S\otimes \mtc{O}_S(\Gamma)$, where $\Gamma$ is a smooth conic contained in $S$.

\begin{lemma}\label{notunigonal}
    Let $(S,h)$ be a unigonal K3 surface of degree $22$. Then $S$ cannot be a degeneration of a $(4,22)$-polarized K3 surface.
\end{lemma}

\begin{proof}
Let $\mts{H}$ be the irreducible subscheme of a Hilbert scheme such that a general point in $\mts{H}$ parametrizes a $(4,22)$-polarized K3 surface, and $\mts{H}^{\circ}$ be the open subset of $\mts{H}$ consisting of points parametrizing K3 surfaces $X$ with at worst du Val singularities such that the degree $22$ polarization extends to $X$. Consider the natural forgetful morphism $\tau:\mts{H}^{\circ}\rightarrow \mts{F}_{22}$ to the moduli space of polarized K3 surfaces of degree $22$, whose image is an 18-dimensional irreducible subvariety of $\mts{F}_{22}$. Notice that the locus of unigonal degree $22$ polarized K3 surfaces is of codimension $1$ in $\mts{F}_{22}$. Therefore, if the image of $\tau$ intersects the locus of unigonal K3 surfaces, then there will be a closed sub-locus of $\tau(\mts{H}^{\circ})$ of codimension $1$ parametrizing unigonal K3 surfaces. In particular, there must be a unigonal K3 surface with at worst an $A_1$-singularity deforming to $(4,22)$-polarized K3 surfaces. Thus it suffices to show that any unigonal K3 surface of degree $22$ which is either smooth or has one $A_1$-singularity cannot be the degeneration of a family of $(4,22)$-polarized K3 surfaces.

Suppose that a $(4,22)$-polarized K3 surface degenerates to a smooth unigonal K3 surface $(S,L)$ of degree $22$ along a DVR $R$ essentially of finite type over $\bC$. Let $\eta$ (resp. $\xi$) be the generic (resp. closed) point in $\Spec R$. 

Up to a finite base change of $R$ we may assume that the generic fiber is isomorphic to a smooth quartic K3 surface $S_{\eta}$ in $\mb{P}^3_{k(\eta)}$ containing a smooth conic $\Gamma_{\eta}$. Let $H_{\eta}$ be the unique hyperplane section containing $\Gamma_{\eta}$, and $Q_{\eta}$ be a general quadric section which contains $\Gamma_{\eta}$. Let $\Gamma'_{\eta}$ (resp. $C_{\eta}$) be the residue curve of $\Gamma_{\eta}$ with respect to $H_{\eta}$ (resp. $L_{\eta}$). Let $L_{\eta}$ be the degree $22$ polarization on $S_{\eta}$. The we have the linear equivalence relations $$C_{\eta}\ \sim \ 2H_{\eta}-\Gamma_{\eta} \ \sim \ 
 H_{\eta}+\Gamma'_{\eta} \ \sim \ \Gamma_{\eta}+2\Gamma'_{\eta}\ ,$$ and $$L_{\eta}\ \sim  \ 4H_{\eta}-C_{\eta} \ \sim \  2H_{\eta}+\Gamma_{\eta} \ \sim \ 3\Gamma_{\eta}+2\Gamma'_{\eta} \ .$$ 
 Moreover, the $\Gamma_{\eta}$ and $\Gamma_{\eta}'$ satisfy that $$(\Gamma_{\eta}^2)\ =\ ({\Gamma'_{\eta}}^2)\ =\ 2,\ \ \textup{and} \ \ (\Gamma_{\eta}.\Gamma'_{\eta})\ =\ 4.$$

 On $S=S_{\xi}$, one has that $L\sim 12F+\Sigma$, where $F$ and $\Sigma$ are the class of a fiber and the negative section respectively. Let $\Gamma$ (resp. $\Gamma'$) be the degeneration of $\Gamma_{\eta}$ (resp. $\Gamma'_{\eta}$) on the central fiber $S$. Then one has the relation $$(L.\ \Gamma)=(L_{\eta}.\ \Gamma_{\eta})=2,\ (L.\ \Gamma')=12-4=8,\ (L.\ \Sigma)=10,\ (L.\ F)=1,$$ and $(L.D)\geq 12$ for any other irreducible curve $D$. As a consequence, one has to have $\Gamma\sim 2F$, which contradicts with $(\Gamma^2)=-2$.

 Now assume that a $(4,22)$-polarized K3 surface degenerates to a unigonal K3 surface $(S,L)$ of degree $22$ with an $A_1$-singularity $p$, along a DVR $R$. Keep the same notations as above. Let $F_0$ be the fiber containing $p$, and $\pi:\wt{S}\rightarrow S$ be the blow-up at $p$. Let $F_1$ be the strict transform of $F_0$, and $E$ be the exceptional divisor. The surface $(S_{\eta},L_{\eta})$ degenerates to $(\wt{S},\wt{L}:=\pi^{*}L)$. Thus one has that $(\wt{L}.E)=0$ and $(\wt{L}.F_1)=1$. Denote by $\wt{\Gamma}$ and $\wt{\Gamma'}$ the degeneration of $\Gamma_\eta$ and $\Gamma_{\eta}'$ to $\wt{S}$. Set $\Gamma\sim a\cdot F_1+b\cdot E$ and $\Gamma'\sim (10-a)\cdot F_1+c\cdot E$, then $(L.\Gamma)=2$ implies that $a=2$. By the relation $(\Gamma^2)=-2$, one has either $b=1$ or $b=3$.  If $b=1$, then $$4\ =\ (\Gamma.\ \Gamma')\ =\ (-E.\ 8F_1+cE)\ =\ -16+2c,$$ and hence $c=10$. This contradicts with $(\Gamma')^2=-2$. Similarly, if $b=3$, then $$4\ =\ (\Gamma.\ \Gamma')\ =\ (E.\ 8F_1+cE)\ =\ 16-2c,$$ and hence $c=6$, which is also impossible.
    
\end{proof}

\section{K-semistable limits of a one-parameter family}\label{KSSLIMITS}

In this section, we will study the K-semistable $\mb{Q}$-Gorenstein limits of the Fano threefolds in the deformation family №2.15. Our goal is to prove the following characterization of such limits.

\begin{theorem}\label{23complete}
Let $X$ be a K-semistable $\bQ$-Fano variety that admits a $\bQ$-Gorenstein smoothing to smooth Fano threefolds in the family \textnumero 2.15. Then  $X$ is Gorenstein canonical and isomorphic to the blow-up of $\mb{P}^3$ along a $(2,3)$-complete intersection curve.
\end{theorem}

\subsection{Volume comparison and general elephants}

In this subsection, we use local-to-global volume comparison from \cite{Fuj18, Liu18, LX19, Liu22} and general elephants on (weak) Fano threefolds  \cite{Sho79, Rei83} (see also \cite{Amb99, Kaw00}) to show that every K-semistable limit $X$ of family \textnumero 2.15 is Gorenstein canonical whose anti-canonical linear system $|-K_X|$ is base-point-free.

\begin{theorem}[cf. \cite{LX19}]\label{nonvanishing}
Let $X$ be a $\bQ$-Gorenstein smoothable K-semistable (weak) $\mb{Q}$-Fano threefold with volume $V:=(-K_X)^3\geq 20$. Then the followings hold.
\begin{enumerate}
    \item The variety $X$ is Gorenstein canonical;
    \item There exists a divisor $S\in |-K_X|$ such that $(X,S)$ is a plt pair, and that $(S,-K_X|_S)$ is a (quasi-) polarized K3 surface of degree $V$.
    \item If $D$ is a $\mb{Q}$-Cartier Weil divisor on $X$ which deforms to a $\bQ$-Cartier Weil divisor on a $\bQ$-Gorenstein smoothing of $X$, then $D$ is Cartier. 
\end{enumerate}
\end{theorem}

\begin{proof}
Our proofs of (1) and (3) are very similar to \cite{LX19}.

(1) Since $X$ is K-semistable, then for any point $x\in X$, by Theorem \ref{volume} we have $$\wh{\vol}(x,X)\ \geq\  20\cdot \left(\frac{3}{4}\right)^3\ >\ 8.$$ Assume to the contrary that $x\in X$  has local Gorenstein index  $d\geq 2$. 
Let $\pi:(\tilde{x}\in\wt{X})\rightarrow (x\in X)$ be the index one cover of $K_X$. If $\tilde{x}\in \wt{X}$ is not smooth, then it follows from \cite[Theorem 1.3(1)]{LX19} and \cite[Theorem 1.3]{XZ21} that $$8 d\  <\ d\cdot \wh{\vol}(x,X)\ =\ \wh{\vol}(\tilde{x},\wt{X}) \ \leq\  {16},$$ and hence $d<2$, which is a contradiction. Thus $\wt{x}\in \wt{X}$ is a smooth point which implies $\wh{\vol}(\wt{x}, \wt{X}) = 27$, and  a similar inequality implies that $8d\leq 27$, i.e. $d\leq 3$. 

In this case,  $x\in X$ is a $\bQ$-Gorenstein smoothable quotient singularity of order $2$ or $3$. If $d=2$, then $x\in X$ is of type $\frac{1}{2}(1,1,0)$, since the $\frac{1}{2}(1,1,1)$-singularity is not smoothable (cf. \cite[Theorem 3]{Sch71}). Therefore, the canonical divisor $K_X$ is Cartier at $x$, a contradiction. If $d=3$, then $x\in X$ is of type $\frac{1}{3}(1,1,0)$ or $\frac{1}{3}(1,2,0)$, as otherwise $x\in X$ is an isolated quotient singularity which is not smoothable (cf. \cite[Theorem 3]{Sch71}). Suppose that $x\in X$ is of type $\frac{1}{3}(1,1,0)$, then there is a 1-dimensional singular locus of $X$ near $x$. Let $\mts{X}\rightarrow (0\in T)$ be a $\mb{Q}$-Gorenstein smoothing of $X$ over a pointed smooth curve $0\in T$ such that $\mts{X}_0\simeq X$. Taking a general hyperplane section in $\mts{X}$ through $x$, we get a $\bQ$-Gorenstein smoothing of a surface singularity of type $\frac{1}{3}(1,1)$, which contradicts the classification of $T$-singularities (cf. \cite[Proposition 3.10]{KSB88}). Therefore, $x\in X$ is of type $\frac{1}{3}(1,2,0)$, and $K_X$ is Cartier.

(2) Now by Theorem \ref{generalelephant}, we have $H^0(X,-K_X)\neq0$, and $(S,-K_X|_S)$ is a polarized K3 surface of degree $V$ for a general member $S\in |-K_X|$.

(3) For the last statement, we approach via the similar local volume estimates argument. Let $ \mts{X}\rightarrow (0\in T)$ be a $\mb{Q}$-Gorenstein smoothing of $X$ over a pointed smooth curve $0\in T$ such that $D$ deforms to a $\mb{Q}$-Cartier Weil divisor $\mts{D}$ on $\mts{X}$ with $\mts{D}_0=D$. Assume that $D$ is not Cartier at a point $x\in {\mts{X}}_0$. As $\mtc{O}_{{\mts{X}}}({\mts{D}})$ is Cohen-Macaulay (cf. \cite[Corollary 5.25]{KM98}), we have that $\mtc{O}_{{\mts{X}}}({\mts{D}})\otimes \mtc{O}_{{\mts{X}_0}}\simeq \mtc{O}_{{\mts{X}}_0}({\mts{D}}_0)$. Then ${\mts{D}}$ cannot be Cartier at $x$, since otherwise $\mtc{O}_{\mts{X}_0}(\mts{D}_0)$ is invertible at $x$.
    
    We claim that $x\in {\mts{X}}_0$ is a hypersurface singularity. In fact, the same argument in the proof of (1) shows that $x$ is a singularity of type either $\frac{1}{2}(1,1,0)$ or $\frac{1}{3}(1,2,0)$, i.e. locally analytically defined by $(z_1^2+z_2^2+z_3^2=0)\subseteq (0\in\mb{C}^4)$ or $(z_1^2+z_2^2+z_3^3=0)\subseteq (0\in\mb{C}^4)$ .

    As $x\in {\mts{X}}_0$ is a hypersurface singularity, the embedding dimension $\edim({\mts{X}}_0,x):= \dim_{\bC}\mtf{m}_{{\mts{X}}_0,x}/\mtf{m}^2_{\mts{X}_0,x}=4$. Since $\mts{X}_0$ is a Cartier divisor in $\mts{X}$, one has $\edim({\mts{X}}_0,x)\leq 5$, and hence $x\in {\mts{X}}$ is a hypersurface singularity as well. Let $H$ be a general hyperplane section of ${\mts{X}}$ through $x$. Then $x\in H$ is a normal isolated hypersurface singularity, and there is a well-defined $\mb{Q}$-Cartier Weil divisor ${\mts{D}}|_{H}$ on $H$ which satisfies that $$\mtc{O}_{{\mts{X}}}({\mts{D}})\otimes \mtc{O}_{H}\simeq \mtc{O}_{H}({\mts{D}}|_{H}).$$ By the local Grothendieck–Lefschetz theorem (cf. \cite[Main Theorem ii)]{Rob76}), the local class group $\Pic(x\in H)$ is torsion-free, which implies that is Cartier at $x$.
    
\end{proof}

\begin{theorem}\label{generalelephant}\textup{(General elephants, cf. \cite{Rei83,Sho79})}
    Let $X$ be a Gorenstein canonical weak Fano threefold. Then $|-K_X|\neq \emptyset$, and a general element $S\in|-K_X|$ is a K3 surface with at worst du Val singularities.
\end{theorem}

\begin{lemma}\label{deform}
    Let $X$ be a Gorenstein canonical (weak) Fano threefold and $\pi:\mts{X}\rightarrow T$ be a $\mb{Q}$-Gorenstein smoothing of $X=\mts{X}_0$. Then for any K3 surface $S\in|-K_{X}|$ \textup{(}e.g. $S$ is a general member in $|-K_{X}|$\textup{)}, after possibly shrinking $T$, there exists a family of K3 surface $\mts{S}\subseteq \mts{X}$ over $T$ such that $\mts{S}_t\in|-K_{\mts{X}_t}|$ and $\mts{S}_0=S$. In particular, $\mts{S}$ is a Cartier divisor in $\mts{X}$.
\end{lemma}

\begin{proof}
    Consider the exact sequence $$0\longrightarrow \mtc{O}_{\mts{X}}(-K_{\mts{X}}-\mts{X}_0)\longrightarrow \mtc{O}_{\mts{X}}(-K_{\mts{X}})\longrightarrow \mtc{O}_{\mts{X}_0}(-K_{\mts{X}_0})\longrightarrow 0.$$ Applying $\pi_{*}$, one obtains by the Kawamata-Viehweg vanishing theorem that the restriction $$H^0(\mts{X},\mtc{O}_{\mts{X}}(-K_{\mts{X}}))\longrightarrow H^0(\mts{X}_0,\mtc{O}_{\mts{X}_0}(-K_{\mts{X}_0}))$$ is surjective. Moreover, if $S$ is a general member in $|-K_{\mts{X}_0}|$, then one can choose $\mts{S}$ to be general, and hence $\mts{S}_t\in |-K_{\mts{X}_t}|$ is general for any general point $t\in T$.
    
\end{proof}

\subsection{Modifications on the family and degree 6 K3 surfaces}\label{limitkss}

From now on, let $X$ be a K-semistable $\bQ$-Fano variety that admits a $\bQ$-Gorenstein smoothing $\pi:\mts{X}\rightarrow T$  over a smooth pointed curve $0\in T$ such that $\mts{X}_0\simeq X$ and every fiber $\mts{X}_t$ over $t\in T\setminus \{0\}$ is a smooth Fano threefold in the family №2.15. Up to a finite base change, we may assume that the restricted family $\mts{X}^{\circ}\rightarrow T^{\circ}:=T\setminus \{0\}$ is isomorphic to $\Bl_{\mts{C}^{\circ}}(\mb{P}^3\times T^{\circ})$, where $\mts{C}^{\circ}\rightarrow T^{\circ}$ is a family of smooth $(2,3)$-complete intersection curves in $\mb{P}^3$. Let $\mts{E}^{\circ}\subseteq \mts{X}^{\circ}$ be the Cartier divisor corresponding to the exceptional divisor of the blow-up, and $\wt{\mts{Q}}^{\circ}$ be the strict transform of the family of quadric surfaces $\mts{Q}^{\circ}$ containing $\mts{C}^{\circ}$. Let $\mtc{H}^{\circ}$ be the line bundle on $\mts{X}^{\circ}$ which is the pull-back of $\mtc{O}_{\mb{P}^3}(1)$, and $\mtc{L}^{\circ}:=3\mtc{H}^{\circ}-\mts{E}^{\circ}$. Then $\mts{E}^{\circ}$ (resp. $\mtc{H}^{\circ}$, $\mtc{L}^{\circ}$, $\mts{Q}^{\circ}$) extends to $\mts{E}$ (resp. $\mtc{H}$, $\mtc{L}$, $\mts{Q}$) on $\mts{X}$ as a Weil divisor on $\mts{X}$ by taking Zariski closure. We denote by $L$ (resp. $H$, $E$, $Q$) the restriction of $\mtc{L}$ (resp. $\mtc{H}$, $\mts{E}$, $\mts{Q}$) on the central fiber $\mts{X}_0\simeq X$.

\begin{prop}\label{bpf}
    The complete linear system $|-K_{X}|$ is base-point-free. 
\end{prop}
    
\begin{proof}
Since $(-K_X)^3 = (-K_{\mts{X}_t})^3 = 22\geq 20$, we can apply Theorem \ref{nonvanishing}(1) and get that  $-K_X$ is an ample Cartier divisor.
Let $S\in|-K_{X}|$ be a general member. Then by Theorem \ref{nonvanishing}(2), $(S,-K_{X}|_S)$ is a polarized (possibly singular) K3 surface of degree $22$. Consider the short exact sequence $$0\longrightarrow \mtc{O}_{X}\longrightarrow\mtc{O}_{X}(-K_{X})\longrightarrow \mtc{O}_{S}(-K_{X}|_{S})\longrightarrow 0.$$ By Kawamata-Viehweg vanishing theorem, we have that $H^1(X,\mtc{O}_{X})=0$, and hence the induced long exact sequence yields a surjection $$H^0(X,-K_{X})\twoheadrightarrow H^0(S,-K_{X}|_S).$$ It follows that $|-K_{X}|_S|$ is base-point-free as by Lemma \ref{notunigonal} $S$ is not unigonal, and hence so is $|-K_{X}|$.

\end{proof}

Let $\theta:\mts{Y}\rightarrow \mts{X}$ be a small $\mb{Q}$-factorialization of $\mts{X}$. Since $\mts{X}$ is klt and $K_{\mts{Y}}=\theta^* K_{\mts{X}}$, we know that $\mts{Y}$ is $\bQ$-factorial and klt. Thus $-K_{\mts{Y}}=\theta^* (-K_{\mts{X}})$ is trivial over $\mts{X}$ and hence nef and big over $\mts{X}$ which implies that $\mts{Y}$ is weak Fano over $\mts{X}$. Therefore, $\mts{Y}$ is of Fano type over $\mts{X}$.
By \cite{BCHM} we can run a minimal model program (MMP) for $\theta^{-1}_{*}\mts{Q}$ on $\mts{Y}$ over $\mts{X}$. As a result, we obtain a log canonical model $\mts{Y} \dashrightarrow \wt{\mts{X}}$ (i.e. an ample model of $\theta^{-1}_{*}\mts{Q}$ over $\mts{X}$) that fits into a commutative diagram
$$\xymatrix{
 & \wt{\mts{X}} \ar[rr]^{f} \ar[dr]_{\wt{\pi}}  &  &  \mts{X} \ar[dl]^{\pi}\\
 & & T  &\\
 }$$ satisfying the following conditions:
\begin{enumerate}
    \item $f$ is a small contraction, and is an isomorphism over $T^{\circ}$;
    \item $-K_{\wt{\mts{X}}}=f^{*}(-K_{\mts{X}})$ is a  $\wt{\pi}$-big and $\wt{\pi}$-nef Cartier divisor; 
    \item $\wt{\mts{X}}_0$ is a $\bQ$-Gorenstein smoothable Gorenstein canonical weak Fano variety whose anti-canonical model is isomorphic to $\mts{X}_0$;
    \item $\wt{\mts{Q}}:=f^{-1}_* \mts{Q}$ is an $f$-ample Cartier Weil divisor (cf. Theorem \ref{nonvanishing}(3)); and
    
    \item $\wt{\mtc{L}}:=f^{-1}_*\mtc{L}\sim_{\bQ, T}\frac{1}{2}(-K_{\wt{\mtc{X}}} + \wt{\mts{Q}})$ is a Cartier divisor (cf. Theorem \ref{nonvanishing}(3)), and $2\wt{\mtc{L}}-(1-\epsilon)\wt{\mts{Q}}\sim_{\bR, T} -K_{\wt{\mtc{X}}} + \epsilon \wt{\mts{Q}}$ is $\wt{\pi}$-ample, for any real number $0<\epsilon\ll1$.
\end{enumerate}

To ease our notation, we denote by $$\wt{X}:=\wt{\mts{X}}_0,\ \ \  g = f|_{\wt{X}}: \wt{X} \to X,\ \ \ \wt{\mtc{H}}:=f_{*}^{-1} \mtc{H},\ \ \ \textup{and}\ \ \ \wt{\mts{E}}:=f_{*}^{-1} \mts{E}.$$ By Theorem \ref{nonvanishing}(3) and linear equivalences $$\wt{\mtc{H}}\sim_{T} \wt{\mtc{L}} - \wt{\mts{Q}}\ \ \ \textup{and} \ \ \ \wt{\mts{E}}\sim_{T} 2\wt{\mtc{L}} - 3\wt{\mts{Q}},$$ we know that both $\wt{\mtc{H}}$ and $\wt{\mts{E}}$ are Cartier. We also denote by $\wt{L}$ (resp. $\wt{H}$, $\wt{E}$, $\wt{Q}$) the restriction of $\wt{\mtc{L}}$ (resp. $\wt{\mtc{H}}$, $\wt{\mts{E}}$, $\wt{\mts{Q}}$) to $\wt{\mts{X}}_0 = \wt{X}$. 

In the remaining of this subsection, our goal is to show that $\wt{\mtc{L}}$ is relatively big and semiample over $T$ (see Proposition \ref{nefness}). A key ingredient is our previous study on the moduli of certain lattice-polarized K3 surfaces (see Lemma \ref{samefamily}).

\begin{lemma}\label{quasipolarized}
The linear series $|-K_{\wt{X}}|$ is base point free. Moreover, if $\wt{S}\in|-K_{\wt{X}}|$ is a K3 surface (e.g. when $S$ is a general member), then $(\wt{S},\wt{L}|_{\wt{S}})$ is a quasi-polarized degree $6$ K3 surface.
\end{lemma}

\begin{proof}

The first statement follows from Proposition \ref{bpf} and the fact that $K_{\wt{X}}=g^{*}K_{X}$. For any $0<\epsilon \ll 1$, since $\wt{\mts{Q}}$ is $\wt{f}$-ample, then we know that $2\wt{\mtc{L}}-(1-\epsilon)\wt{\mts{Q}}$ is $\wt{\pi}$-ample, and hence $(2\wt{L}-(1-\epsilon)\wt{Q})|_{\wt{S}}$ is ample on the K3 surface $\wt{S}$. 

We claim that $\wt{S}$ is represented by a point in $\mts{F}_{\Lambda_2}$ (cf. Section \ref{k33}). Indeed, by deforming to a family of K3 surface in $|-K_{\mts{X}_t}|$ (cf. Lemma \ref{deform}), one sees that the divisors $\wt{L}|_{\wt{S}}$ and $\wt{Q}|_{\wt{S}}$ 
generate a primitive sublattice of $\Pic(\wt{S})$ which is isometric to $\Lambda_2$ where $(2\wt{L}-(1-\epsilon)\wt{Q})|_{\wt{S}}$ corresponds to the vector $h_2$.

By Lemma \ref{samefamily}, we know that $\wt{S}$ is represented by a point in $\mts{F}_{\Lambda_1}$ as well. In particular, $(\wt{L}-\epsilon\wt{Q})|_{\wt{S}}$ is also ample for any $0<\epsilon \ll 1$, and hence $\wt{L}|_{\wt{S}}$ is nef. The degree of $\wt{L}|_{\wt{S}}$ is invariant under deformation, as $\wt{\mtc{L}}$ is a Cartier divisor on $\wt{\mts{X}}$. 
    
\end{proof}


    

\begin{lemma}\label{isomorphismonsection}
    For any K3 surface $\wt{S}\in |-K_{\wt{X}}|$, The restriction map $$H^0(\wt{X},\mtc{O}_{\wt{X}}({\wt{L}}))\longrightarrow H^0(\wt{S},\mtc{O}_{\wt{S}}({\wt{L}|_{\wt{S}}}))$$ is an isomorphism. In particular, we have that $h^0(\wt{X},\mtc{O}_{\wt{X}}({\wt{L}}))=5$. 
\end{lemma}

\begin{proof}
    Let $\wt{S}\in|-K_{\wt{X}}|$ be a K3 surface. By Lemma \ref{quasipolarized} we see that $(\wt{S},\wt{L}|_{\wt{S}})$ is a quasi-polarized degree $6$ K3 surface, and hence $$h^0(\wt{S},\wt{L}|_{\wt{S}})\ =\ \frac{1}{2}({\wt{L}}|_{\wt{S}})^2+2\ =\ 5.$$
    Since $\wt{S}\sim -K_{\wt{X}}$ is Cartier, we have a short exact sequence 
    \[
    0 \to \mtc{O}_{\wt{X}}({\wt{L}}- \wt{S}) \to \mtc{O}_{\wt{X}}({\wt{L}}) \to \mtc{O}_{\wt{S}}({\wt{L}}|_{\wt{S}})\to 0.
    \]
    Notice that ${\wt{L}}-\wt{S}\sim \wt{L}+K_{\wt{X}}\sim -\wt{H}$ is not effective. Hence by taking the long exact sequence we see that 
    $H^0(\wt{X},\mtc{O}_{\wt{X}}({\wt{L}}))\hookrightarrow H^0(\wt{S},{\wt{L}}|_{\wt{S}})$ is injective, and thus $h^0(\wt{X},\mtc{O}_{\wt{X}}({\wt{L}}))\leq 5$. On the other hand, by upper semi-continuity, we have that $$h^0(\wt{X},\mtc{O}_{\wt{X}}({\wt{L}}))\ \geq\  h^0(\wt{\mts{X}}_t,\mtc{O}_{\wt{\mts{X}}_t}({\wt{\mtc{L}}_t}))\ =\ 5$$ for a general $t\in T$. Therefore, one has $h^0(\wt{X},\mtc{O}_{\wt{X}}({\wt{L}}))=5$, and the restriction map is an isomorphism.
    
\end{proof}

\begin{lemma}\label{except}
    We have that $h^0\big(\wt{X},\mtc{O}_{\wt{X}}(\wt{Q})\big)=1$.
\end{lemma}

\begin{proof}
    It suffices to show that $h^0(\wt{X},\mtc{O}_{\wt{X}}(\wt{Q}))\leq1$ as $\wt{Q}$ is effective. For a general K3 surface $\wt{S}\in |-K_{\wt{X}}|$, we have that $|2\wt{L}|_{\wt{S}}|$ is base-point-free and induces a birational morphism contracting $\wt{Q}|_{\wt{S}}$, since the intersection number $(\wt{L}|_{\wt{S}}.\wt{Q}|_{\wt{S}})=0$. Thus $h^0(\wt{S},\mtc{O}_{\wt{S}}(\wt{Q}))=1$. As $\wt{Q}-\wt{S}\sim -2\wt{H}$ is not effective,  one has an injection  $H^0(\wt{X},\mtc{O}_{\wt{X}}(\wt{Q}))\hookrightarrow H^0(\wt{S},\mtc{O}_{\wt{S}}(\wt{Q}))$, which proves the statement.

\end{proof}

\begin{lemma}\label{isomorphismonsection2}
    For any K3 surface $\wt{S}\in |-K_{\wt{X}}|$, The restriction map $$H^0(\wt{X},\mtc{O}_{\wt{X}}(2{\wt{L}}))\longrightarrow H^0(\wt{S},\mtc{O}_{\wt{S}}(2{\wt{L}}|_{\wt{S}}))$$ is surjective. In particular, we have that $h^0(\wt{X},\mtc{O}_{\wt{X}}(2{\wt{L}}))=15$. 
\end{lemma}

\begin{proof}
    Let $\wt{S}\in|-K_{\wt{X}}|$ be a K3 surface. By Lemma \ref{quasipolarized} we see that $(\wt{S},{\wt{L}}|_{\wt{S}})$ is a  quasi-polarized K3 surface of degree $6$, and hence $$h^0(\wt{S},2{\wt{L}}|_{\wt{S}})\ =\ \frac{1}{2}(2{\wt{L}}|_{\wt{S}})^2+2\ =\ 14.$$
    By similar arguments to the proof of Lemma \ref{isomorphismonsection}, we have a long exact sequence
    \[
    0 \to H^0(\wt{X}, \mtc{O}_{\wt{X}}(2{\wt{L}} - \wt{S})) \to H^0(\wt{X}, \mtc{O}_{\wt{X}}(2{\wt{L}})) \to H^0(\wt{S}, \mtc{O}_{\wt{S}}(2{\wt{L}|_{\wt{S}}})).
    \]
    Notice that $2{\wt{L}}-\wt{S}\sim 2{\wt{L}}+K_{\wt{X}}\sim \wt{Q}$ and that $h^0(\wt{X},\mtc{O}_{\wt{X}}(\wt{Q}))=1$ (cf. Lemma \ref{except}). Hence the above long exact sequence implies that $$h^0(\wt{X},\mtc{O}_{\wt{X}}(2{\wt{L}}))\ \leq\  h^0(\wt{S},2{\wt{L}}|_{\wt{S}})+1\ =\ 15.$$ On the other hand, by upper semi-continuity, we have that $$h^0(\wt{X},\mtc{O}_{\wt{X}}(2{\wt{L}}))\ \geq\  h^0(\wt{\mts{X}}_t,\mtc{O}_{\wt{\mts{X}}_t}(2{\wt{\mtc{L}}_t}))\ =\ 15$$ for any general $t\in T$. Thus the proof is finished.
    
\end{proof}

\begin{prop}\label{nefness} The Cartier divisor $\wt{\mtc{L}}$ is $\wt{\pi}$-semiample and  $\wt{\pi}$-big.

\end{prop}

\begin{proof}
    We first prove that $\wt{L}$ is a nef divisor. Recall that we have $2\wt{L}\sim -K_{\wt{X}}+\wt{Q}$. As $-K_{\wt{X}}$ is big and nef, and $\wt{Q}$ is $g$-ample, then we have that $(\wt{L}.C)>0$ for any $g$-exceptional curve $C\subseteq \wt{X}$.

    We claim that the base locus of the linear series $|2\wt{L}|$ is either some isolated points, or is contained in the $g$-exceptional locus, so $\wt{L}$ is nef. By Lemma \ref{quasipolarized}, we know that $|2\wt{L}|_{\wt{S}}|$ is base-point-free, where $\wt{S}\in|-K_{\wt{X}}|$ is a general member. Suppose that $\wt{C}\subseteq \Bs|2\wt{L}|$ is a curve which is not contracted by $g$. Then the intersection $\wt{C}\cap \wt{S}$ is non-empty and consists of finitely many points, which are all base points of $|2\wt{L}|_{\wt{S}}|$ (cf. Lemma \ref{isomorphismonsection2}). This leads to a contradiction.

     Since $\wt{L}= \wt{\mtc{L}}|_{\wt{\mts{X}}_0}$ is nef, and $\wt{\mtc{L}}|_{\wt{\mts{X}}_t}$ is nef for any $t\in T\setminus\{0\}$ as $\wt{\mts{X}}_t\simeq \mts{X}_t$ is a smooth Fano threefold in the family \textnumero 2.15, we conclude that $\wt{\mtc{L}}$ is $\wt{\pi}$-nef. This implies the $\wt{\pi}$-semiampleness of $\wt{\mtc{L}}$ by Kawamata--Shokurov base-point-free theorem, as $\wt{\mts{X}}$ is of Fano type over $T$. Since $\wt{\mtc{L}}|_{\wt{\mts{X}}_t}$ is big for a general $t$, we get the $\wt{\pi}$-bigness of $\wt{\mtc{L}}$.
    
\end{proof}

\subsection{Birational model as a singular cubic threefold}

From the last subsection we know that $\wt{\mtc{L}}$ is a  $\wt{\pi}$-semiample and  $\wt{\pi}$-big Cartier divisor on $\wt{\mts{X}}$
(cf. Proposition \ref{nefness}). Thus taking its ample model over $T$ yields a birational morphism $\phi:\wt{\mts{X}}\to \mts{V}$ that fits into a commutative diagram $$\xymatrix{
 & \wt{\mts{X}} \ar[rr]^{\phi \quad} \ar[dr]_{\wt{\pi}}  &  &  \mts{V}  
 \ar[dl]^{\pi_{\mts{V}}} \\
 & & T&\\
 }$$ 
 Here we have $$\mts{V} \ :=\ \Proj_{T}\bigg(\bigoplus_{m\in \bN} \wt{\pi}_* \left(\wt{\mtc{L}}^{\otimes m}\right)\bigg).$$
 For any $0\neq t\in T$, the morphism $\wt{\mts{X}}_t\rightarrow \mts{V}_t$ contracts precisely the smooth quadric surface $\wt{\mts{Q}}_t$ to an $A_1$-singularity of $\mts{V}_t$, which can be embedded into $\mb{P}^4$ as a singular cubic threefold by the line bundle $(\phi_{*}\mtc{L})|_{\mts{V}_t}$.

Now let us consider the restriction of the morphism $\phi$ to the central fiber $$\phi_0:\wt{\mts{X}}_0=\wt{X}\longrightarrow V:=\mts{V}_0.$$ Let $L_V:=(\phi_{0})_*\wt{L}$ be the $\mb{Q}$-Cartier Weil divisor on $V$. Our main goal in this subsection is to show that $V$ is also a singular cubic threefold, and $\wt{X}\simeq X$ is the blow-up of $V$ at a singular point. These results will help us prove Theorem \ref{23complete}.

\begin{lemma}\label{lem:contract-Q}
    The central fiber $V$ of $\pi_{\mts{V}}$ is a normal projective variety. Moreover, 
    the morphism $\phi_0: \wt{X}\to V$ is birational, contracts $\wt{Q}$ to a singular point $P$ of $V$, and is an isomorphism on $\wt{X}\setminus \wt{Q}$.
\end{lemma}

\begin{proof}
    We first show that $V$ is normal and  $\phi_0$ is birational. Since both $\wt{\mtc{L}}$ and $-K_{\wt{\mts{X}}}$ are nef and big over $T$ and $\wt{\mts{X}}$ is klt, then Kawamata--Viehweg vanishing theorem implies that $$R^i \wt{\pi}_* \wt{\mtc{L}}^{\otimes m} \ =\  0$$ for any $i> 0$ and $m\in \bN$. Thus by cohomology and base change, the sheaf $\wt{\pi}_* \wt{\mtc{L}}^{\otimes m}$ is locally free and satisfies that $$\left(\wt{\pi}_* \wt{\mtc{L}}^{\otimes m}\right)\otimes k(0)\ \simeq\ H^0(\wt{X}, \wt{L}^{\otimes m}).$$ As a result, one has $$V\ =\ \mts{V}_0\ \simeq\ \Proj \bigg(\bigoplus_{m\in \bN} H^0(\wt{X}, \wt{L}^{\otimes m})\bigg)$$ is the ample model of $\wt{L}$ on $\wt{X}$, which implies normality of $V$ and birationality of $\phi_0$. 
    
    Consider the restriction $\phi|_{\wt{\mts{Q}}}:\wt{\mts{Q}}\rightarrow \mts{P}$ of $\phi$ to $\wt{\mts{Q}}$. As $\mts{P}$ is an irreducible scheme over $T$ and the general fiber of $\mts{P}\rightarrow T$ is a point, then $\mts{P}$ is a curve, and $\mts{P}_0$ is a single point. On the other hand, if $C\subseteq\wt{X}$ is a curve such that $(C.\wt{L})=0$, then $C\subseteq \wt{Q}$ because $2\wt{L}-(1-\epsilon)\wt{Q}$ is ample for $0<\epsilon \ll 1$. Thus the last statement is proved.
    
\end{proof}

\begin{prop}\label{Gorensteincan}
    The variety $V$ is a Gorenstein canonical Fano variety. Moreover, the $\bQ$-Cartier Weil divisor $L_V$ is Cartier on $V$, and $-K_V\sim 2L_V$.
\end{prop}

\begin{proof}
    From the linear equivalence $-K_{\wt{X}}\sim 2\wt{L}-\wt{Q}$ and Lemma \ref{lem:contract-Q}, we know that $-K_{V}=(\phi_0)_*(-K_{\wt{X}})\sim 2L_V$ is an ample $\mb{Q}$-Cartier divisor. Thus  $V$ is a $\bQ$-Fano variety as $\phi_0^* K_V = K_{\wt{X}} - \wt{Q}\leq K_{\wt{X}}$ and $X$ is klt. 
    
    Next, we show that $2L_V$ is Cartier which implies that $V$ is a Gorenstein canonical Fano variety.
    Let $\wt{S}\in |-K_{\wt{X}}|$ be a general member, and $S_V:=(\phi_0)_* \wt{S}$ be its pushforward as a Weil divisor on $V$. By Theorem \ref{generalelephant} we know that $(\wt{X}, \wt{S})$ is a plt log Calabi--Yau pair. Thus $(V, S_V)$ is also a plt log Calabi--Yau pair, which implies that $S_V$ is normal and $\phi_0|_{\wt{S}}:\wt{S}\to S_V$ is birational. As a result,  $(S_V,L_V|_{S_V})$ is a degree $6$ polarized K3 surface as it is the ample model of the quasi-polarized K3 surface  $(\wt{S},\wt{L}|_{\wt{S}})$  of degree 6. In particular, $L_V|_{S_V}$ is an ample Cartier divisor on $S_V$, and $2L_V|_{S_V}$ is base-point-free by Theorem \ref{Mayer}. By Lemma \ref{isomorphismonsection2}, we have that  $$H^0(V,\mtc{O}_{V}(2L_V))\ \simeq\ H^0(\wt{X},\mtc{O}_{\wt{X}}(2\wt{L}))\ \longrightarrow \ H^0(\wt{S},\mtc{O}_{\wt{S}}(2\wt{L}|_{\wt{S}}))\ \simeq \ H^0(S_V,\mtc{O}_{S_V}(2L_V|_{S_V}))$$ is surjective, and hence $|2L_V|$ is base-point-free. 
    Thus $2L_V$ is Cartier.

    Finally we show that $L_V$ is Cartier.
    By Lemma \ref{lem:contract-Q} we know that  $\phi_0$ induces an isomorphism between $\wt{X}\setminus \wt{Q}$ and $V\setminus\{P\}$. Since $\wt{L}$ is Cartier on $\wt{X}$, it suffices to show that $L_V$ is Cartier at $P$. As $L_V$ is a $\mb{Q}$-Cartier Weil divisor and $V$ is klt, we know that $\mtc{O}_{V}(L_V)$ is a Cohen-Macaulay divisorial sheaf by \cite[Corollary 5.25]{KM98}. Since $S_V\sim 2L_V$ is Cartier, the restriction $\mtc{O}_{V}(L_V)\otimes_{\mtc{O}_V}\mtc{O}_{S_V}$ is also Cohen-Macaulay. Thus one has a natural isomorphism $$\mtc{O}_{V}(L_V)\otimes_{\mtc{O}_V}\mtc{O}_{S_V}\ \simeq\ (\mtc{O}_{V}(L_V)\otimes_{\mtc{O}_V}\mtc{O}_{S_V})^{**} \ \simeq \ \mtc{O}_{S_V}(L_V|_{S_V}).$$ As a consequence, one has $$\dim_k(\mtc{O}_{V}(L_V)\otimes k(P))=\dim_k(\mtc{O}_{S_V}(L_V|_{S_V})\otimes k(P))=1$$ because $L_V|_{S_V}$ is a Cartier divisor on $S_V$. Therefore, the sheaf $\mtc{O}_{V}(L_V)$ is invertible, and hence $L_V$ is Cartier.

\end{proof}

\begin{corollary} \label{cor:V-cubic}
    The linear system $|L_V|$ on $V$ is very ample and embeds $V$ into $\mb{P}^4$ as a Gorenstein canonical cubic threefold. 
\end{corollary}

\begin{proof}
    We proved that $V$ is a Gorenstein canonical Fano variety with $-K_V\sim 2L_V$, where $L_V$ is a Cartier divisor (cf. Proposition \ref{Gorensteincan}). Moreover, we have $(L_V^3) = (\wt{L}^3) = (\wt{\mtc{L}}_t^3) = 3$ for a general $t\in T$.  Hence the statement follows immediately from \cite{Fuj90}.
    
\end{proof}


\begin{lemma}\label{lem:doublept}
    The point $P$ is a double point of $V$.
\end{lemma}

\begin{proof}
    We first show that $P$ is a singular point of the cubic threefold. Indeed, the morphism $\phi$ contracts $\wt{Q}$, and the image of $\wt{Q}_t$ for a general $t\in T$ is an ordinary double point $\mts{P}_t$ of $\mts{V}_t$. Thus $P$ is also a singular point of $V=\mts{V}_0$ by closedness of singular locus. 
    
    Now it suffices to prove that the multiplicity $\mult_PV$ cannot be $3$. For a general element $\wt{S}\in|-K_{\wt{X}}|$, the image ${S}_V:=\phi_0(\wt{S})$ is normal, and hence it is the ample model of $\wt{L}|_{\wt{S}}$. In particular, $({S}_V,L_V|_{{S}_V})$ is a degree $6$ polarized K3 surface and hence a $(2,3)$-complete intersection in $\mb{P}^4$. Moreover, since $K_{\wt{\mts{X}}}$ and $\wt{\mts{Q}}$ are both Cartier, the intersection number $(K_{\mts{X}_t}\cdot \wt{\mts{Q}}_t^2)$ is constant in $t\in T$, which implies that $(\wt{S} \cdot \wt{Q}^2) = -2$. Thus $\wt{S}\cap \wt{Q}$ is non-empty which implies that $P\in S_V$ by Lemma \ref{lem:contract-Q}.  If $P$ is a triple point of $V$, then $\mult_P{S}_V\geq3$, which contradicts the fact that $P$ is at worst a double point of ${S}_V$.
    
\end{proof}

\begin{prop}\label{prop:blow-up-iso}
    There exist natural isomorphisms $$\Bl_{\mts{P}}\mts{V}\ \simeq\  \wt{\mts{X}}\ \simeq\ \mts{X}$$ over $T$, where $\mts{P}=\phi(\wt{\mts{Q}})$. In particular, we have $\Bl_P V \simeq \wt{X}\simeq X$.
\end{prop}

\begin{proof}
    We first show that the blow-up is compatible with taking fibers. As both $\wt{\mtc{L}}$ and $-K_{\wt{\mts{X}}}$ are nef and big over $T$ and $\wt{\mts{X}}$ is klt, then the Kawamata--Viehweg vanishing theorem implies that $$R^i \wt{\pi}_* \wt{\mtc{L}} \ =\  0$$ for any $i> 0$. Thus by cohomology and base change, the sheaf $\wt{\pi}_* \wt{\mtc{L}}$ is locally free of rank $4$. By shrinking the base $T$, we may assume that $T=\Spec R$, where $R$ is a DVR with parameter $t$, and $\mts{V}\subseteq \mb{P}^4\times T$. For simplicity, we assume that $\mts{P}\subseteq \mb{P}^4\times T$ is defined by the ideal $(x_1,...,x_4)$, and $\mts{V}$ is defined locally near $\mts{P}$ by the polynomial $f(\underline{x},t):=f_2(\underline{x},t)+f_3(\underline{x},t)$, where $\underline{x}=(x_1,...,x_4)$, $f_i(\underline{x},t)$ is a homogeneous polynomial in variables $\underline{x}$ with coefficients in $R$ such that $\deg_{\underline{x}}f_i(\underline{x},t)=i$. Moreover, by Lemma \ref{lem:doublept} we have that $f_2(\underline{x},t)$ is non-zero for any $t$. In order to prove that $\Bl_{\mts{P}_0}\mts{V}_0\simeq (\Bl_\mts{P}\mts{V})|_{t=0}$, it suffices to show that the natural morphism $$\frac{(x_1,x_2,x_3,x_4)^k}{t(x_1,...,x_4)^k+(f)} \ \twoheadrightarrow\ \left(\frac{(x_1,x_2,x_3,x_4,t)}{(t,f)}\right)^k$$ is an isomorphism for any $k\gg0$. Indeed, for any polynomial $g(\underline{x},t)$ such that the degree of each monomial with respect to $\underline{x}$ is at least $k$, if $g(\underline{x},0)$ is divisible by $f(\underline{x},0)$, then $g(\underline{x},0)=f(\underline{x},0)\cdot h(\underline{x})$ for some polynomial $h$ of degree at least $k-2$. Consequently, one has $$g(\underline{x},t)-f(\underline{x},t)\cdot h(\underline{x})=t\cdot l(\underline{x},t)$$ such that the $\underline{x}$-degree of each monomial in $l(\underline{x},t)$ is at least $k$ as desired.

    As an immediate result, we see that $\Bl_{\mts{P}}\mts{V}$ is normal. In fact, since $\Bl_PV=\Bl_{\mts{P}_0}\mts{V}_0$ is an integral scheme, then it satisfies (R$_0$) and (S$_1$) conditions. Thus $\Bl_{\mts{P}}\mts{V}$ satisfies (R$_1$) and (S$_2$) conditions, which is equivalent to being normal.

    Notice that we have the desired isomorphism over $T^{\circ}$. Thus $\Bl_{\mts{P}}\mts{V}$ and $\wt{\mts{X}}$ are birational and isomorphic in codimension one. Let $\mts{Q}'$ be the exceptional divisor of the blow-up $\psi:\Bl_{\mts{P}}\mts{V}\rightarrow \mts{V}$, and $\mtc{L}'=\psi^{*}\phi_* \wt{\mtc{L}}\sim_T \psi^* \mtc{O}_{\mts{V}}(1)$. From the blow-up construction we know that the $\mb{Q}$-Cartier $\bQ$-divisor $\mtc{L}'-\epsilon\mts{Q}'$ is ample over $T$ for any rational number $0<\epsilon\ll1$. On the other hand, since $\wt{\mtc{L}}$ is nef over $T$ (cf. Proposition \ref{nefness}), and $2\wt{\mtc{L}}-(1-\epsilon)\wt{\mts{Q}}$ is ample over $T$ for any rational number $0<\epsilon\ll 1$ (from the construction of $\wt{\mts{X}}$), we know that  $\wt{\mtc{L}}-\epsilon\wt{\mts{Q}}$ is also ample over $T$. 
    Since $\Bl_{\mts{P}}\mts{V}$ and $\wt{\mts{X}}$  are isomorphic in codimension $1$, and $\wt{\mtc{L}}-\epsilon\wt{\mts{Q}}$ is the birational transform of $\mtc{L}'-\epsilon \mts{Q}'$, we conclude that $\Bl_{\mts{P}}\mts{V}\simeq  \wt{\mts{X}}$. As a result, we see that $-K_{\wt{\mts{X}}}$ is ample over $T$, and hence $\wt{\mts{X}}\simeq\mts{X}$.
    
\end{proof}

\begin{remark}\label{rem:cubic-blowup-fano}
    \textup{Let $P$ be a point in $\mb{P}^4$, and $\pi:\Bl_P\mb{P}^4\rightarrow \mb{P}^4$ be the blow-up. Set $L:=\mtc{O}_{\mb{P}^4}(1)$ to be the pull-back of the class of hyperplane section and $E$ to be the class of exceptional divisor. If $V$ is a singular cubic threefold in $\mb{P}^4$ with a double point $P$, then the blow-up $\wt{V}$ of $V$ at $P$ admits a natural ample line bundle $(2L-E)|_{\wt{V}}$, which coincides with the anti-canonical bundle of $\wt{V}$ by adjunction. If moreover $V$ has Gorenstein canonical singularities and the projective tangent cone of $V$ at $P$ is a normal quadric surface, then by inversion of adjunction, the polarized variety $(\wt{V},(2L-E)|_{\wt{V}})$ is an anti-canonically polarized Gorenstein canonical Fano variety.}
\end{remark}

Before proving the main theorem in this section, we show that the set of blow-ups of $\mb{P}^3$ along $(2,3)$-complete intersection curves has a one-to-one correspondence to the set of blow-ups of integral singular cubic threefolds at  double points. This is a generalization of the smooth case in Section \ref{sec:2.15}.

\begin{lemma}\label{lem:Sarkisov-singular}
    There is a one-to-one correspondence between the set of $(2,3)$-complete intersection curves up to projective equivalence in $\mb{P}^3$, and the set of integral singular cubic threefold with a double point up to projective equivalence in $\mb{P}^4$. Moreover, this one-to-one correspondence is given by a Sarkisov link structure: the blow-up of $\mb{P}^3$ along a $(2,3)$-complete intersection curve is identified with the blow-up of an integral singular cubic threefold at a double point.
\end{lemma}

\begin{proof}
    Let $V$ be the singular cubic threefold in $\mb{P}^4_{x_0,...,x_4}$ defined by the polynomial $$f(x_0,x_1,...,x_4)\ = \ x_0f_2(x_1,...,x_4)-f_3(x_1,...,x_4),$$ which acquires a double point at $P=[1,0,...,0]$. Let $C$ be a $(2,3)$-complete intersection curve in $\mb{P}^3_{x_1,...,x_4}$ defined by the polynomials $f_2(x_1,...,x_4)$ and $f_3(x_1,...,x_4)$. The ideal of $C$ has a natural resolution $$0\ \longrightarrow \mtc{O}_{\mb{P}^3}(-5)\ \stackrel{\binom{f_3}{-f_2}}{\longrightarrow} \ \mtc{O}_{\mb{P}^3}(-2)\oplus \mtc{O}_{\mb{P}^3}(-3) \  \stackrel{(f_2,f_3)}{\longrightarrow} \ I_C \ \longrightarrow \ 0,$$ which yields a natural embedding of $\Bl_C\mb{P}^3$ into $$ \Bl_P\mb{P}^4\ \simeq\  \mb{P}_{\mb{P}^3}(\mtc{O}\oplus\mtc{O}(-1))\ \simeq\  \mb{P}_{\mb{P}^3}(\mtc{O}(-2)\oplus\mtc{O}(-3))\ =:\ \bf{P}.$$ The ideal of $\Bl_C\mb{P}^3$ in the projective bundle ${\bf{P}}\stackrel{\pi}{\longrightarrow}\mb{P}^3$ is generated by the composition $$\pi^{*}\mtc{O}_{\mb{P}^3}(-5)\ \stackrel{\binom{f_3}{-f_2}}{\longrightarrow} \ \pi^{*}\mtc{O}_{\mb{P}^3}(-2)\oplus \pi^{*}\mtc{O}_{\mb{P}^3}(-3) \ \stackrel{(u,v)}{\longrightarrow} \ \mtc{O}_{\bf{P}}(1),$$ where $(u,v)$ is the evaluation map coming from the Euler sequence of projective bundles. In particular, on the patch $(u\neq0)$, the variety $\Bl_C\mb{P}^3$ is defined by the polynomial $tf_2(z_1,...,z_4)-f_3(z_1,...,z_4)$, where $t=\frac{v}{u}$. This coincides with the generator (up to a constant) of the ideal $I_{\Bl_PV/\Bl_P\mb{P}^4}$. 

    If $C$ and $C'$ are two projectively isomorphic $(2,3)$-complete intersection curves, then it is clear that their corresponding cubic threefolds are also projectively equivalent in $\mb{P}^4$. Conversely, suppose that $V$ and $V'$ are two projectively equivalent singular cubic threefolds with double points $P$ and $P'$ respectively. After projective transformations we may assume that $P=P'=[1,0,...,0]$. Then their defining equations are given by 
    \begin{align*}
    f(x_0,x_1,...,x_4)&\ = \ x_0f_2(x_1,...,x_4)-f_3(x_1,...,x_4),\ \textup{and}\\ f'(x_0,x_1,...,x_4)&\ = \ x_0f'_2(x_1,...,x_4)-f'_3(x_1,...,x_4),
    \end{align*}
    respectively. The projectively equivalence implies that by applying a $\PGL(4)$-action on the coordinate $[x_1,...,x_4]$ followed by a translation $x_0 \mapsto x_0+l(x_1,...,x_4)$, where $l(x_1,...,x_4)$ is a suitable linear function in $x_1,...,x_4$, one can recover $f$ from $f'$. By our construction of the correspondence, this comes down to saying that the ideals of $C$ and $C'$ are projectively equivalent.
    
    Under this correspondence, the curve $C$ is a complete intersection if and only if $f_2(x_1,...,x_4)$ and $f_3(x_1,...,x_4)$ have no common factor, which is equivalent to saying that the cubic threefold $V=\mb{V}(f(x_0,...,x_4))$ is integral. 
    
\end{proof}


\begin{proof}[Proof of Theorem \ref{23complete}]
    By Theorem \ref{nonvanishing}(1) we know that $X$ is Gorenstein canonical.
    By Corollary \ref{cor:V-cubic}, Lemma \ref{lem:doublept} and Proposition \ref{prop:blow-up-iso}, we have $X\simeq \Bl_p V$ where $V$ is a Gorenstein canonical cubic threefold and $P\in V$ is a double point. Then the statement follows directly from Lemma \ref{lem:Sarkisov-singular}.
    
    
\end{proof}

\section{K-moduli, VGIT, and the Hassett–Keel program}\label{VGITHK}

In this section, we will finish the proof of our main theorem, that is, the K-moduli space $\ove{M}^K_{\textup{№2.15}}$ coincides with a Hassett--Keel model of the moduli space $\ove{M}_4$ of genus four curves. The key observation comes from Theorem \ref{23complete}: every Fano threefold $X$ parametrized by $\ove{M}^K_{\textup{№2.15}}$ is the blow-up of $\mb{P}^3$ along a $(2,3)$-complete intersection curve.

\subsection{Computation of CM line bundles}\label{computationcm}

We follow the same notations in Section \ref{VGIT}

\begin{lemma}\label{lem:nonflat-blowup}
    Let $\mts{Z}\hookrightarrow \mb{P}^3\times \mb{P}^9\times\mb{P}^{19}$ be the scheme-theoretic intersection of pull-backs of universal families of  quadric surfaces and  cubic surfaces. In other words, $\mts{Z}$ is the total space of the (non-flat) family of $(2,3)$-intersection subschemes of $\bP^3$ over $\mb{P}^9\times\mb{P}^{19}$. Then $\mts{Z}$ 
    is smooth. 
    
    Moreover, the restricted family $\mts{Z}|_V$ over the open subset $$V\ :=\ \big\{([q],[g])\in \mb{P}^9\times\mb{P}^{19}\ :\  g \textup{ is not divisible by }q\big\} $$ descends to a (non-flat) family $\mts{Y}$ over $\mb{P}\mtc{E}$ via the $\mb{A}^4$-bundle $$V\ \longrightarrow \ \mb{P}\mtc{E}.$$ In particular, the total space $\mts{X}:=\Bl_{\mts{Y}}(\mb{P}^3\times\mb{P}\mtc{E})$ over $\mb{P}\mtc{E}$ is smooth.
    
\end{lemma}

\begin{proof}
   First notice that $\mts{Z}$ is a complete intersection in $\mb{P}_{\underline{x}}^3\times \mb{P}_{\underline{y}}^9\times\mb{P}_{\underline{z}}^{19}$ of codimension two, given by the two equations $$\big(y_0x_0^2+\cdots+y_9x_3^2=0,\ z_0x_0^3+\cdots+z_{19}x_3^3=0\big),$$ whose Jacobian is always of rank $2$. Therefore, the total space $\mts{Z}$ is a smooth.

   Observe that $\mb{P}\mtc{E}$ is a quotient projective bundle of the trivial projective bundle $\mb{P}^9\times\mb{P}^{19}$ (cf. exact sequence \ref{*}). Indeed, let $W$ be the complement of $V$ in $\mb{P}^9\times\mb{P}^{19}$, and the linear projection $p:V\rightarrow \mb{P}\mtc{E}$ from the center $W$ is given by $$([q],[g])\ \mapsto \ ([q],[\ove{g}]),$$ which is an $\mb{A}^4$-bundle. Fiberwise, over each point $[Q]\in\bP^9$, the projection is exactly the restriction map $$p_{[Q]}\ :\ H^0(\bP^3,\mtc{O}_{\bP^3}(3))\ \longrightarrow\ H^0(Q,\mtc{O}_{Q}(3))$$ sending every section $g$ to its restriction $\ove{g}:=g|_{Q}.$
   
   For any point $\omega=([q],[\ove{g}])\in\mb{P}\mtc{E}$, the fiber $p^{-1}(\omega)$ is the projective space parametrizing polynomials $g\in \mb{P}H^0(\mb{P}^3,\mtc{O}_{\mb{P}^3}(3))$ whose restrictions on the quadric $\mb{V}(q)$ is $[\ove{g}]$. The pull-back of the (non-flat) universal family $\mts{Y}$ over $\mb{P}\mtc{E}$ coincides with the restriction $\mts{Z}|_V$. It follows that $\mts{Y}$ is smooth and hence so is $\mts{X}$.

\end{proof}

Let $U\subseteq \mb{P}\mtc{E}$ be the big open subset parameterizing $(2,3)$-complete intersection curves. Then Lemma \ref{lem:nonflat-blowup} shows that there is a non-flat family $i:\mts{Y}\hookrightarrow\mb{P}^3 \times \mb{P}\mtc{E}$ of $(2,3)$-intersection subschemes of $\bP^3$ over $\mb{P}\mtc{E}$, whose restriction over $U$ is the universal family of $(2,3)$-complete intersection curves $i_U:\mathscr{C}\hookrightarrow \mb{P}^3\times U$.

Denote by $p_1,p_2$ the projection maps from $\mb{P}^3\times \bP\mtc{E}$ to the two factors respectively, and $f:=p_2\circ \rho:\mts{X}\rightarrow \bP\mtc{E}$ the composition.

\begin{lemma}\label{lem:non-ci-codim}
    With the above notation, we have $\codim_{\mts{Y}} (\mts{Y}\setminus \mts{C}) \geq 3$.
\end{lemma}

\begin{proof}
Consider the family $p_2\circ i: \mts{Y}\to \bP\mtc{E}$. Since $\mts{Y}$ is smooth and irreducible and a general fiber of $p_2\circ i$ is a $(2,3)$-complete intersection curve, we have 
\[
\dim \mts{Y} = \dim \bP\mtc{E} + 1 = 25.
\]
Next we consider the restriction $\mts{Y}\setminus \mts{C}\to \bP\mtc{E}\setminus U$ of $p_2\circ i$. By definition we know that $\bP\mtc{E}\setminus U$ parameterizes $([q], [\ove{g}])$ such that $q$ and $g$ share a common linear factor $l$. It is clear that the parameter space of non-integral $[q]$ has dimension $6$, and the fiber of $\bP\mtc{E}\setminus U$  over every non-integral $[q]$ has dimension $h^0(Q, \mtc{O}_Q(2)) -1=8$ with $Q = \mathbb{V}(q)$. Thus we have
\[
\dim \bP\mtc{E}\setminus U  = 6 + 8 = 14.
\]
Since every fiber of $\mts{Y}\setminus \mts{C}\to \bP\mtc{E}\setminus U$ is contained in the quadric surface, we have 
\[
\dim \mts{Y}\setminus \mts{C} \leq \dim \bP\mtc{E}\setminus U +2 = 16. 
\]
The proof is finished as $25-16\geq 3$.
\end{proof}

Let $j:\mts{E}\hookrightarrow \mts{X}$ be the exceptional divisor of the blow-up, and the induced map $\rho_{\mts{E}}:\mts{E}\rightarrow \mts{Y}$ is a $\mb{P}^1$-bundle. We set $\zeta:=\mtc{O}_{\mts{E}}(1) = \mtc{O}_{\mts{X}}(-\mts{E})|_{\mts{E}}$ and $\gamma:=c_1(N_{\mts{Y}/\mb{P}^3\times \bP\mtc{E}})$.

Let $\mts{Q}\subseteq \mb{P}^3\times U$ be the pull-back of the universal quadric surface $\cQ$ over $\mb{P}^9$ along the natural projection $U\rightarrow \mb{P}^9$. Then $\mts{C}\subseteq \mts{Q}$ and we have an exact sequence of normal bundles $$0\longrightarrow N_{\mts{C}/\mts{Q}}\longrightarrow N_{\mts{C}/\mb{P}^3\times U}\longrightarrow N_{\mts{Q}/\mb{P}^3\times U}|_{\mts{C}}\longrightarrow0.$$ Since $\mts{Q}$ comes from the universal quadric, then one has $$N_{\mts{Q}/\mb{P}^3\times U}|_{\mts{C}}\ \simeq\ \mtc{O}_{\mb{P}^3\times U}(\mts{Q})|_{\mts{C}}\ \simeq\  (2h+\eta)|_{\mts{C}}.$$ 
Now we compute the normal bundle $N_{\mts{C}/\mts{Q}}\simeq \mtc{O}_{\mts{Q}}(\mts{C})|_{\mts{C}}$. By restricting to the fiber over a general point $[Q]\in \mb{P}^9$, we see that $$\mtc{O}_{Q\times \mb{P}H^0(Q,\mtc{O}_{Q}(3))}(\mts{C}_Q)\ =\ (3h+\xi)|_{Q\times \mb{P}H^0(Q,\mtc{O}_{Q}(3))}.$$ 
To compute the contribution of $\eta$, we restrict to the fiber over a general point $p\in\mb{P}^3$. The inclusions $\mts{C}\subseteq \mts{Q}\subseteq \mb{P}^3\times U$ come down to saying that $$\{\ (Q,[S])\ :\ p\in Q\cap S\ \}\ \subseteq\  \{\ (Q,[S])\ |\  p\in Q\}\ \subseteq\  U,$$ which implies that $\eta$ does not contribute to the class of $\mts{C}$ in $\mts{Q}$. Thus we deduce that the normal bundle $N_{\mts{C}/\mts{Q}}=(3h+\xi)|_{\mts{C}}$. It follows that $$c_1(N_{\mts{C}/\mb{P}^3\times U})=(5h+\xi+\eta)|_{\mts{C}},\quad \textup{and}\quad c_2(N_{\mts{C}/\mb{P}^3\times U})=(3h+\xi)\cdot(2h+\eta)|_{\mts{C}}.$$
Moreover, the codimension $2$ cycle $[\mts{C}]$  of $\bP^3\times U$ satisfies
$[\mts{C}] = (3h+\xi)\cdot (2h+\eta)|_{\bP^3\times U}$.
By Lemma \ref{lem:non-ci-codim}, we have 
\[
[\mts{Y}] = (3h+\xi)\cdot (2h+\eta),
\]
\[
\gamma = c_1(N_{\mts{Y}/\mb{P}^3\times \bP\mtc{E}})=(5h+\xi+\eta)|_{\mts{Y}}, \quad \textup{and}\quad c_2(N_{\mts{Y}/\mb{P}^3\times \bP\mtc{E}})=(3h+\xi)\cdot(2h+\eta)|_{\mts{Y}}.
\]
As a result, we have that 
\begin{equation}\nonumber
\begin{split}
f_{*}(\rho^{*}p_1^{*}\mtc{O}_{\mb{P}^3}(1)^2.\mts{E}^2)&\ =\ -p_{2*}\rho_{*}(\rho^{*}p_1^{*}\mtc{O}_{\mb{P}^3}(1)^2.j_{*}\zeta)\ = \ -p_{2*}\rho_{*}j_{*}(\rho_{\mts{E}}^{*}i^{*}p_1^{*}\mtc{O}_{\mb{P}^3}(1)^2.\zeta)\\
&\ =\ -p_{2*}i_{*}\rho_{\mts{E}*}(\rho_{\mts{E}}^{*}i^{*}p_1^{*}\mtc{O}_{\mb{P}^3}(1)^2.\zeta)\ =\ -p_{2*}i_{*}(i^{*}p_1^{*}\mtc{O}_{\mb{P}^3}(1)^2.\rho_{\mts{E}*}\zeta)\\
&\ =\ -p_{2*}i_{*}(i^{*}p_1^{*}\mtc{O}_{\mb{P}^3}(1)^2)\ = -p_{2*}(p_1^{*}\mtc{O}_{\mb{P}^3}(1)^2.[\mts{Y}])\\
&\ =\ -p_{2*} (h^2 \cdot (3h+\xi) (2h+\eta))\\ & =\ -(2\xi+3\eta),
\end{split}
\end{equation}

\begin{equation}\nonumber
\begin{split}
f_{*}(\rho^{*}p_1^{*}\mtc{O}_{\mb{P}^3}(1).\mts{E}^3)&\ =\ p_{2*}\rho_{*}(\rho^{*}p_1^{*}\mtc{O}_{\mb{P}^3}(1).j_{*}(\zeta^2))\ =\ p_{2*}\rho_{*}j_{*}(\rho_{\mts{E}}^{*}i^{*}p_1^{*}\mtc{O}_{\mb{P}^3}(1).\zeta^2)\\
&\ =\ p_{2*}i_{*}\rho_{\mts{E}*}(\rho_{\mts{E}}^{*}i^{*}p_1^{*}\mtc{O}_{\mb{P}^3}(1).\zeta^2)\ =\ p_{2*}i_{*}(i^{*}p_1^{*}\mtc{O}_{\mb{P}^3}(1).\rho_{\mts{E}*}(\zeta^2))\\
&\ =\ -p_{2*}i_{*}(i^{*}p_1^{*}\mtc{O}_{\mb{P}^3}(1).\gamma)\ =\ -p_{2*}(p_1^{*}\mtc{O}_{\mb{P}^3}(1).i_{*}\gamma)\\
&\ =\ -p_{2*}(h \cdot (5h+\xi+\eta) \cdot (3h+\xi) (2h+\eta)) \\
&\ =\ -(16\xi+21\eta),
\end{split}
\end{equation}

\begin{equation}\nonumber
\begin{split}
f_{*}(\mts{E}^4)&\ =\ -p_{2*}\rho_{*}j_{*}(\zeta^3)\ =\ -p_{2*}i_{*}\rho_{\mts{E}*}(\zeta^3)\\
&\ =\ -p_{2*}i_{*}((c_1(N_{\mts{Y}/\mb{P}^3\times \bP\mtc{E}})^2-c_2(N_{\mts{Y}/\mb{P}^3\times \bP\mtc{E}})))\\
&\ =\ -p_{2*}(((5h+\xi+\eta)^2-(3h+\xi)(2h+\eta)).[\mts{Y}])\\
&\ =\ -p_{2*}(((5h+\xi+\eta)^2-(3h+\xi)(2h+\eta))\cdot (3h+\xi)(2h+\eta))\\
&\ =\ -(86\xi+99\eta).
\end{split}
\end{equation}

\begin{prop}
    With the same notation as above, one has $$-f_{*}\big((-K_{\mts{X}/\mb{P}\mtc{E}})^4\big)= 22\xi+51\eta.$$
\end{prop}

\begin{proof}
    Summing up the computations above, one has that 
\begin{equation}\label{CMLB}
\begin{split}
 -f_{*}((-K_{\mts{X}/\mb{P}\mtc{E}})^4) &\ =\ -f_{*}((-\rho^{*}p_1^{*}K_{\mb{P}^3}-\mts{E})^4)\\
&\ =\ -f_{*}((4\rho^{*}p_1^{*}\mtc{O}_{\mb{P}^3}(1)-\mts{E})^4)\\
&\ =\ -\big(-(86\xi+99\eta)+16(16\xi+21\eta)-96(2\xi+3\eta)\big)\\
&\ =\ 22\xi+51\eta.
\end{split}
\end{equation}
\end{proof}

\begin{prop}\label{prop:CMslope}
    Let $f_U:\mts{X}_U\rightarrow U$ be the restriction of $\mts{X}\rightarrow \mb{P}\mtc{E}$ over $U$.     Denote by $\eta_U$ and $\xi_U$ the restriction of $\eta$ and $\xi$ to $U$ respectively. 
    Then the followings hold.
    \begin{enumerate}
        \item Every fiber $\mts{X}_t$ of $f_U$ for $t\in U$ is isomorphic to $\Bl_{\mts{C}_t}\mb{P}^3$, in particular, it is integral and Gorenstein with ample anti-canonical divisor.
        \item The morphism $f_U:\mts{X}_U\rightarrow U$ is flat and projective, and $-K_{\mts{X}/U}$ is an $f_U$-ample Cartier divisor that commutes with base change.
        \item For the anti-canonically polarized family $f_U: (\mts{X}_U, -K_{\mts{X}_U/U})\to U$, we have $$\lambda_{\CM,f_U}\ =\ 22\xi_U+51\eta_U.$$ In particular, the CM $\mb{Q}$-line bundle $\lambda_{\CM,f_U}$ is ample on $U$.
    \end{enumerate}
\end{prop}

\begin{proof}
(1) Since every fiber of $\mts{C}$ over $U$ by definition is a $(2,3)$-complete intersection curve in $\bP^3$, taking Rees algebra associated to the ideal sheaf $I_{\mts{C}}$ commutes with taking fibers. Then the statement follows from the fact that taking $\Proj$ is commutative with taking fibers. By Remark \ref{rem:cubic-blowup-fano} we know that  $-K_{\mts{X}_t}$ is an ample line bundle, and $\mts{X}_t$ is Gorenstein for each $t\in U$.

(2) The flatness of $f$ follows from part (1) and miracle flatness (ref. \cite[00R4]{Sta18}) as both $\mts{X}$ and $U$ are smooth and $f_U$ has constant fiber dimensions.  The second statement follows from (1).

(3)  Since both $\mts{X}$ and $\mb{P}\mtc{E}$ are smooth projective varieties, by Grothendieck--Riemann--Roch theorem, for $q\gg1$ one has that $$
c_1(f_{*}(-K_{\mts{X}/\mb{P}\mtc{E}}^{\otimes q}))=\frac{q^4}{24}f_{*}(-K_{\mts{X}/\mb{P}\mtc{E}})^4 +\frac{q^3}{12}f_{*}(-K_{\mts{X}/\mb{P}\mtc{E}})^4 + O(q^2).$$ Since CM line bundles are functorial, by similar
arguments to \cite[Proposition 2.23]{ADL19}, one has that $$c_1\big(\lambda_{\CM,f_U}\big)\ =\  -f_{U*}(-K_{\mts{X}_U/U})^4 \ =\ 22\xi_U+51\eta_U.$$
Since $U$ is a big open subset of $\mb{P}\mtc{E}$ (i.e. the complement of $U$ in $\mb{P}\mtc{E}$ has codimension at least $2$),  we have a natural isomorphism $\Pic(\mb{P}\mtc{E})\simeq \Pic(U)$, and we will not distinguish ($\mb{Q}$-)line bundles on $U$ and those on $\mb{P}\mtc{E}$. We define the \emph{slope} of a $\mb{Q}$-line bundle $a\eta+b\xi$ with $a\neq 0$ to be $t=\frac{b}{a}$. By  \cite[Theorem 2.7]{Ben14}, the nef cone of $\mb{P}\mtc{E}$ has extremal rays of slope $0$ and $\frac{1}{2}$. Therefore, the CM $\mb{Q}$-line bundle $\lambda_{\CM,f_U}$ is ample on $U$.

\end{proof}

\begin{lemma}\label{lem:quadric-normal}
    Let $C$ be a $(2,3)$-complete intersection in $\mb{P}^3$. Suppose $X:=\Bl_C\mb{P}^3$ is a K-semistable $\bQ$-Fano variety. Then the unique quadric surface $Q$ containing $C$ is normal.
\end{lemma}

\begin{proof}
    Notice that the sub-locus $Z$ of complete intersection curves which is contained in a non-normal quadric $Q$ is irreducible. Thus it suffices to show that for a general member $[C]=(Q,[S])$ in $Z$, the blow-up $X:=\Bl_C\mb{P}^3$ is K-unstable. Observe such a curve $C$ a union of two elliptic curves $C_1\cup C_2$ such that $C_1$ (resp. $C_2$) is contained in a plane $P_1$ (resp. $P_2$), and that $C_1\cap C_2=\{p_1,...,p_3\}\subseteq P_1\cap P_2$. We may assume that $P_1=\bV(x_1)$ and $P_2=\bV(x_2)$, where $x_0,...,x_3$ is a homogeneous coordinate of $\mb{P}^3$. Consider the $\mb{G}_m$-action $\sigma$ of weight $(-1,1,1,-1)$, which induces a normal test configuration $(\mtc{X}, -K_{\mtc{X}/\bA^1})$ of $X$, whose central fiber is $X_0=\Bl_{C_0}\mb{P}^3$. It then follows from Equation (\ref{CMLB}) that $\Fut(\mtc{X}, -K_{\mtc{X}/\bA^1})$ is proportional up to a positive constant to the weight $$\mu^{\lambda_{\CM}}([C_0];\sigma)\ =\ \mu^{t_0}(Q,[S_0];\sigma)\  =\  3t-2,$$ which is negative since $t_0<\frac{2}{3}$ (cf. \cite[Proposition 4.6]{CMJL14}), and hence $X$ is K-unstable.
    
\end{proof}

\subsection{Deformation theory of $\Bl_C\mb{P}^3$}

\begin{prop}\label{deformation}
    Let $X$ be the blow-up of $\bP^3$ along a $(2,3)$-complete intersection curve $C$. Assume in addition that $C$ is contained in a normal quadric surface, and that $X$ has Gorenstein canonical singularities. 
    Then $\Ext^2(\Omega^1_{X},\mtc{O}_X)=0$. In particular, there are no obstructions to deformation of $X$.
\end{prop}

\begin{proof}
    By Lemma \ref{lem:Sarkisov-singular}, we have a  birational morphism $\phi:X\rightarrow V\subseteq \mb{P}^4$ that contracts the strict transform of the quadric surface containing $C$ to a double point $P\in V$, where $V$ is a singular cubic threefold. Then it sits into a commutative diagram $$\xymatrix{
 & X \ar[rr]^{\phi} \ar@{^(->}[d]  &  &  V \ar@{^(->}[d]\\
 & \wt{\mb{P}}:=\Bl_p\mb{P}^4  \ar[rr]^{\quad \psi}  &   & \mb{P}^4. \\
 }$$ By taking $R\Hom(\cdot, \mtc{O}_X)$ of the short exact sequence $$0 \longrightarrow (N_{X/\wt{\mb{P}}})^*\longrightarrow\Omega^1_{\wt{\mb{P}}}|_X \longrightarrow  \Omega^1_X \longrightarrow  0,$$ it suffices to show that $H^1(X,N_{X/\wt{\mb{P}}})=0$ and $\Ext^2(\Omega^1_{\wt{\mb{P}}}|_X,\mtc{O}_X)=0$. 
 

 Let $L:=\psi^{*}\mtc{O}_{\mb{P}^4}(1)$, $F\simeq \mb{P}^3$ be the $\psi$-exceptional divisor, and $Q:=F|_X$ be the $\phi$-exceptional locus, which is a normal quadric surface. Then $N_{X/\wt{\mb{P}}}\simeq \mtc{O}_X(3L-2F)$, which is ample, and hence $H^1(X,N_{X/\wt{\mb{P}}})=0$ by Kawamata-Viehweg vanishing.
 
 To show the second vanishing, let us consider the short exact sequence $$0 \longrightarrow \psi^*\Omega^1_{\mb{P}^4}|_X \longrightarrow\Omega^1_{\wt{\mb{P}}}|_X \longrightarrow  \Omega^1_F|_X \longrightarrow  0,$$ where we use the fact that the local generator of $\mtc{I}_{X/\wt{\mb{P}}}$ is not a zero divisor of $\mtc{O}_F$. It suffices to show that $\Ext^2(\psi^*\Omega^1_{\mb{P}^4}|_X,\mtc{O}_X)=0$ and $\Ext^2(\Omega^1_F|_X,\mtc{O}_X)=0$. Consider the pull-back of the Euler sequence on $\mb{P}^4$ $$0 \longrightarrow \mtc{O}_X \longrightarrow (L|_X)^{\oplus 5 } \longrightarrow  \psi^{*}T_{\mb{P}^4}|_X \longrightarrow  0.$$ One has that $H^3(X,\mtc{O}_X)=0$ and $H^2(X,L|_X)=0$ by Kawamata-Viehweg vanishing, and hence $$\Ext^2(\psi^*\Omega^1_{\mb{P}^4}|_X,\mtc{O}_X)\simeq H^2(X,\psi^{*}T_{\mb{P}^4}|_X)=0.$$ Now we consider the Euler sequence on $F$, viewed as torsion sheaves on $\wt{\mb{P}^4}$ $$0 \longrightarrow \Omega^1_F\longrightarrow \mtc{O}_F(-1)^{\oplus 4 } \longrightarrow   \mtc{O}_F \longrightarrow  0.$$ Twisting by $\mtc{O}_X$, one obtains $$0 \longrightarrow \Omega^1_F|_X\longrightarrow \mtc{O}_F(-1)^{\oplus 4 }|_X \longrightarrow   \mtc{O}_F|_X \longrightarrow  0.$$ We now need to prove that $\Ext^2(\mtc{O}_F(-1)|_X,\mtc{O}_X)=0$ and $\Ext^3(\mtc{O}_F|_X,\mtc{O}_X)=0$. Since $\mtc{O}_F|_X=\mtc{O}_Q$, then by Serre duality one has $$\Ext^3(\mtc{O}_F|_X,\mtc{O}_X)\simeq H^0(X,\omega_X\otimes \mtc{O}_F|_X)^{*} \simeq H^0(Q,\omega_X|_Q)^{*}=0$$ as $\omega_X$ is anti-ample. Similarly, we have that $$\Ext^2(\mtc{O}_F(-1)|_X,\mtc{O}_X)\simeq H^1(X,\omega_X\otimes \mtc{O}_F(-1)|_X)^{*} \simeq H^1(Q,2(F-L)|_Q)^{*}=0$$ since $(F-L)|_Q$ is an anti-ample line bundle.
 
\end{proof}

\begin{corollary}\label{cor:stack-smooth}
    The K-moduli stack $\mts{M}^K_{\textup{№2.15}}$ is smooth, and its good moduli space $\ove{M}^K_{\textup{№2.15}}$ is normal. Moreover,  $\mts{M}^K_{\textup{№2.15}}$ is a smooth connected component of $\mts{M}^K_{3, 22}$.
\end{corollary}

\begin{proof}
For any point $[X]\in \mts{M}^K_{\textup{№2.15}}$, Theorem \ref{23complete} implies that $X$ is Gorenstein canonical and isomorphic to the blow-up of $\bP^3$ along a $(2,3)$-complete intersection curve $C$. By Lemma \ref{lem:quadric-normal}, we know that $C$ is contained in a normal quadric surface as $X$ is K-semistable. Thus the ($\bQ$-Gorenstein) deformation of $X$ is unobstructed by Proposition \ref{deformation}. Therefore, the K-moduli stack $\mts{M}^K_{3,22}$ is smooth at $[X]$, and hence $\mts{M}^K_{\textup{№2.15}}$ is a smooth connected component of $\mts{M}^K_{3,22}$. The normality of $\ove{M}^K_{\textup{№2.15}}$ follows from smoothness of $\mts{M}^K_{\textup{№2.15}}$ by \cite[Theorem 4.16(viii)]{Alp13}. 

\end{proof}

\subsection{VGIT and Hassett--Keel program} Based on the computation in Section \ref{computationcm}, one has the following result.

\begin{lemma}\label{representable}
    Let $C$ be a $(2,3)$-complete intersection curve in $\mb{P}^3$ contained in a normal quadric surface such that $X:=\Bl_C\mb{P}^3$ is a $\mb{Q}$-Fano variety. Then one has a natural isomorphism $\Aut(\mb{P}^3,C)\simeq \Aut(X)$.
\end{lemma}

\begin{proof}
    By the universal property of blow-ups, we can lift an automorphism of $\mb{P}^3$ preserving $C$ to an automorphism of $X$, and hence we have a natural injection $\Aut(\mb{P}^3,C)\hookrightarrow \Aut(X)$. Conversely, let $E$ be the exceptional divisor of the blow-up $\pi:X\rightarrow \mb{P}^3$, and $\wt{Q}$ be the strict transform of the unique normal quadric surface $Q$ containing $C$. Consider the Fano scheme $F_1(X)$ of $(-K_X)$-lines, i.e.  rational curves $l\subseteq X$ satisfying $(-K_X.l)=1$. Let $H:=\pi^{*}\mtc{O}_{\mb{P}^3}(1)$ and $L:=3H-E$ be two Cartier divisors, which are both big and semiample. From Lemma \ref{lem:Sarkisov-singular} we know that the linear system $|L|$ is base-point-free and induces a birational morphism $\psi: X \to V\subset \bP^4$ where $V$ is an integral singular cubic threefold, $\psi$ contracts $\wt{Q}$ to a double point of $V$ and is isomorphic elsewhere. Since we have $-K_X\sim H+ L$,  every line $l$ is either contracted by $\pi$ or contracted by $\psi$. In the former case, $l$ is  a ruling of $\pi|_E:E \rightarrow C$. In the latter case,   $l$ is contained in $\wt{Q}$. Thus the lines in $F_1(X)$ swept out the divisor $E+\wt{Q}$. Moreover, every two distinct rulings in $E$ do not intersect, while there are lines in $\wt{Q}$ intersecting each other. Therefore, any automorphism $\sigma$ of $X$ has to send $\wt{Q}$ to $\wt{Q}$. Since $2H \sim -K_X - \wt{Q}$, we know that $\sigma^*(2H) \sim 2H$.  It follows that $\sigma$ descends to an automorphism $\ove{\sigma}$ of $\mb{P}^3$ which is the ample model of $2H$. Thus $\ove{\sigma}$ preserves $C$ as it is the exceptional locus of $\pi$.
    
\end{proof}

\begin{prop}\label{stackiso}
    There is an isomorphism of stacks $$\mts{M}^K_{\textup{№2.15}}\ \simeq\ [U^K/\PGL(4)],$$ where $U^K$ is the $\PGL(4)$-equivariant open subset of $U$ parametrizing $(2,3)$-complete intersection curves $C$ such that $\Bl_C\mb{P}^3$ is a K-semistable $\bQ$-Fano variety.
\end{prop}

\begin{proof}
    Since $U^K$ parametrizes K-semistable $\mb{Q}$-Fano varieties, then by universality of K-moduli stacks, there exists a morphism $$\varphi:\ [U^K/\PGL(4)]\ \longrightarrow\ \mts{M}^K_{\textup{№2.15}}.$$ To show $\varphi$ is an isomorphism, we will construct the its inverse morphism. By the construction of K-moduli stacks, we know that $$\mts{M}^K_{\textup{№2.15}}\ \simeq\ [T/\PGL(N+1)],$$ where $T\subseteq \Hilb_{\chi}(\mb{P}^{N})$ is a locally closed subscheme of the Hilbert scheme with respect to the Hilbert polynomial $\chi(m):=\chi(X,-mrK_X)$ for $r\in \bZ_{>0}$ sufficiently divisible. By Corollary \ref{cor:stack-smooth}, we know that $T$ is irreducible and smooth. Let $\pi:\mts{Z}\rightarrow T$ be the universal family. By Lemma \ref{lem:quadric-normal} and the proof of Lemma \ref{representable}, there is a unique divisor $\mts{Q}\subseteq \mts{Z}$ relative Cartier over $T$ such that each fiber $\mts{Q}_t$ is the strict transform of the normal quadric surface containing the curve we blow-up along. 
    
    Consider the $\pi$-big and $\pi$-nef line bundle $\mts{R}:=-\omega_{\mts{Z}/T}-\mts{Q}$ on $\mts{Z}$, whose restriction on each fiber $\mts{Z}_t$ is isomorphic to the pull-back of the line bundle $\mtc{O}_{\mb{P}^3}(2)$. Taking the $\pi$-ample model with respect to $\mts{R}$, one obtains a family $\mts{Y}\rightarrow T$ such that each fiber $\mts{Y}_t$ is isomorphic to $\mb{P}^3$. 
    
    Take a fppf cover $\sqcup_i T_i\rightarrow T$ of $T$ such that the pull-back family $p_i:\mts{Y}_i\rightarrow T_i$ is a trivial $\mb{P}^3$-bundle. Let $\mtc{H}=\sqcup_i \mtc{H_i}$ be the line bundle on $\sqcup_i \mts{Y}_i$ which is isomorphic to the pull-back of $\mtc{O}_{\mb{P}^3}(1)$ over each $T_i$. By Kawamata-Viehweg vanishing, we know that $(p_i)_*\mtc{H}_i$ is a rank $4$ vector bundle over $T_i$. Let $\mtc{P}_i/T_i$ be the $\PGL(4)$-torsor induced by projectivized basis of $(p_i)_*\mtc{H}_i$. As the cocycle condition of $\{(p_i)_*\mtc{H}_i/T_i\}_i$ is off by $\pm1$, then we deduce that $\{\mtc{P}_i/T_i\}_i$ is a fppf descent datum descending to a $\PGL(4)$-torsor $\mtc{P}/T$, which is $\PGL(N+1)$-equivariant. This induces a morphism $$\varphi^{-1}:\mts{M}^K_{\textup{№2.15}}\longrightarrow [U^K/\PGL(4)].$$ Since both stacks are smooth, then $\varphi^{-1}$ is indeed the inverse of $\varphi$.
    
\end{proof}

The following is the first main theorem in this section.

\begin{theorem}\label{mainiso}
    Let $C$ be a $(2,3)$-complete intersection curve in $\mb{P}^3$, and $X$ be the blow-up of $\mb{P}^3$ along $C$. Then $X$ is K-(semi/poly)stable if and only if $[C]$ is a GIT-(semi/poly)stable point in $\mb{P}\mtc{E}$ with respect to the linearized ample $\bQ$-line bundle $\eta+t_0\cdot\xi$, where $t_0=\frac{22}{51}$. Moreover, there is a natural isomorphism $$\mts{M}^K_{\textup{№2.15}}\ \simeq\  \mts{M}^{\GIT}(t_0) = \left[\mb{P}\mtc{E}^{\sst}(t_0)/\PGL(4) \right],$$ which descends to an isomorphism between their good moduli spaces $$\ove{M}^K_{\textup{№2.15}}\ \simeq\  \ove{M}^{\GIT}(t_0) = \mb{P}\mtc{E}\sslash_{t_0}\PGL(4).$$
\end{theorem}

To prove Theorem \ref{mainiso}, we want to apply the same argument in \cite{PT06,PT09} using CM line bundles. To be more concrete, to show that a GIT-unstable object is also K-unstable, we can construct a test configuration induced by a destabilizing 1-PS of the object, and use the proportionality of the CM line bundle and the polarization of the GIT. However, since the space $U$ parametrizing $(2,3)$-complete intersections is not compact, the limit of a destabilizing 1-PS could be outside $U$, i.e. not a complete intersection anymore. Therefore, we need a more intricate argument involving wall-crossing structure of VGIT-moduli spaces. Similar arguments appeared in \cite{ADL21}.

\begin{lemma}\label{destabilization}
    Let $C\in|\mtc{O}_{\mb{P}^1\times\mb{P}^1}(3,3)|$ be a curve. We embed $C$ into $\bP^3$ as a $(2,3)$-complete intersection curve contained in a smooth quadric surface $Q$. If $\Bl_C\mb{P}^3$ is a K-semistable $\bQ$-Fano variety, then $[C]$ is a GIT-semistable point in $|\mtc{O}_{\mb{P}^1\times\mb{P}^1}(3,3)|$ under the natural $\Aut(\mb{P}^1\times\mb{P}^1)$-action.
\end{lemma}

\begin{proof}
    We first show that the blow-up $X:=\Bl_C\mb{P}^3$ is a Gorenstein canonical Fano variety for every curve $C$ contained in a smooth quadric surface $Q$. Since this question is local analytically, we may work in local analytic coordinates near $C$ and assume that 
    \[
    C = \mathbb{V}(x, g(y,z))\subset  Q = \mathbb{V}(x) \subset  \bC^3_{x,y,z}.
    \]
    Here $g$ is a non-zero holomorphic function near the origin. Then we may cover $\Bl_C \bC^3$ using two charts $X_1$ and $X_2$, where 
    \[
    X_1 = \mathbb{V}(tx - g(y,z)) \subseteq \bC^4_{x,y,z,t}, \quad \textrm{ and }\quad X_2=\mathbb{V}(x - sg(y,z)) \subseteq \bC^4_{x,y,z,s}.
    \]
    It is clear that $X_1$ has $cA$-singularities hence is  Gorenstein canonical by \cite[Theorem 5.34]{KM98}, and $X_2$ is smooth. Thus $X$ is a Gorenstein canonical Fano variety.

    Next, we apply Theorem \ref{KimpliesGIT} to the family $f_Q: \mts{X}_Q \to \bP H^0(Q, \mtc{O}_Q(3))$ obtained by base change of $f_U:\mts{X}_U\to U$ to  $\bP H^0(Q, \mtc{O}_Q(3)) = \pi^{-1}([Q])\hookrightarrow U$. Conditions (i) and (ii) are satisfied by  Proposition \ref{stackiso}. Condition (iii) on the ampleness of $\lambda_{\CM, f_Q}$ follows from the ampleness of $\lambda_{\CM, f_U}$ (cf. Proposition \ref{prop:CMslope}). 
    Therefore, the result follows from Theorem \ref{KimpliesGIT} and facts that $|\mtc{O}_{\mb{P}^1\times\mb{P}^1}(3,3)|$ has Picard rank one and that $\Aut^0(\bP^1\times\bP^1)$ has no non-trivial character.
    
\end{proof}


    

\begin{lemma}\label{unstable2}
     Let $C\in|\mtc{O}_{\mb{P}(1,1,2)}(6)|$ be a curve. We embed $C$ in $\bP^3$ as a $(2,3)$-complete intersection curve contained in a quadric cone. Then $X=\Bl_C\mb{P}^3$ is K-unstable if  $C$ has a singularity at the cone point of $\mb{P}(1,1,2)$, which is worse than an $A_1$-singularity.
\end{lemma}

\begin{proof}
    If $X$ is not klt then it is automatically K-unstable by \cite{Oda13}. Thus we may assume that $X$ is klt. 
    Let $\mb{P}(1,1,2)$ be defined by the equation $(x_1x_3-x_2^2=0)$. We can choose a test configuration $\mtc{X}$ induced by the 1-PS $\lambda$ of $\PGL(4)$ of weight $(9,1,-3,-7)$. Then one has that $\Fut(\mtc{X}, -K_{\mtc{X}/\bA^1})$ is proportional up to a positive constant to the weight $$\mu^{\lambda_{\CM}}([C_0];\lambda)\ =\ \mu^{t_0}(Q,[S_0];\lambda)\ \leq \ 11t-6,$$ which is negative as $t_0<\frac{6}{11}$ (cf. \cite[Proposition 4.12]{CMJL14})
    
\end{proof}

\begin{proof}[Proof of Theorem \ref{mainiso}]
    We first show that the K-semistability of $X=\Bl_{C}\mb{P}^3$ implies the $t_0$-GIT-semistability of $[C]$. Assume to the contrary that $C$ is a $t_0$-GIT-unstable curve such that $X=\Bl_{C}\mb{P}^3$ is K-semistable. By Lemma \ref{lem:quadric-normal}, one sees that the unique quadric $Q$ containing $C$ is normal. If $Q$ is smooth, then by Lemma \ref{destabilization}, $[C]$ is GIT-semistable under the action of $\Aut(\mb{P}^1\times\mb{P}^1)$ as an element in $|\mtc{O}_{\mb{P}^1\times \mb{P}^1}(3,3)|$. Thus $[C]$ is $t$-GIT-semistable for $t\in (0, \frac{2}{9})$ by Remark \ref{vgitremark}(3). If $Q$ is singular, then by Lemma \ref{unstable2}, either $C$ does not passes through the cone point, or $C$ has an $A_1$-singularity at the cone point. Thus  $C$ admits an isotrivial degeneration to a $T_i$-GIT-polystable curve (either a triple conic or a double conic and two different rulings), where $i=1$ or $2$. In particular, $C$ is $t$-semistable for some $0<t<t_0$. Consider the maximum $T\in (0, \frac{2}{3}]$ such that $C$ is $T$-GIT-semistable. Then one has that $T$ is a VGIT wall such that $0<T<t_0$. Therefore, by the description of VGIT wall crossing (cf. \cite{CMJL14}), there exists a 1-PS $\lambda$ of $\PGL(4)$ such that $\mu^{t_0}([C];\lambda)<0$, $\mu^T([C];\lambda) = 0$ and that $\lim_{s\to 0}\lambda(s)\cdot [C]$ lies in $U$ as it is $T$-GIT-semistable. It follows that up to scaling by a positive constant, one has $$\Fut(\mtc{X},-K_{\mtc{X}/\mb{A}^1})= \mu^{\lambda_{\CM}}([C];\lambda)\ =\ \mu^{t_0}([C];\lambda)\  < \ 0,$$ and hence $X$ is K-unstable, where $\mtc{X}$ is the test configuration of $X$ induced by $\lambda$. This is a contradiction. 

    Next, we show that K-polystability of $X$ implies the $t_0$-GIT-polystability of $[C]$. From the above argument we know that $[C]$ is $t_0$-GIT-semistable. Let $\lambda$ be the $1$-PS of $\PGL(4)$ such that $[C_0]:=\lim_{s\to 0} \lambda(s)\cdot [C]$ is $t_0$-GIT-polystable. Then we know that $[C_0]\in U$ which implies that up to scaling a positive constant, we have
    \[
    \Fut(\mtc{X},-K_{\mtc{X}/\mb{A}^1})= \mu^{\lambda_{\CM}}([C];\lambda)\ =\ \mu^{t_0}([C];\lambda)\  = \ 0.
    \]
    Here $\mtc{X}$ is the test configuration of $X$ induced by $\lambda$ as before. Thus by \cite[Lemma 3.1]{LWX21} we know that $X_0=\Bl_{C_0}\bP^3$ is K-semistable. This implies that $X\cong X_0$ and hence $C\cong C_0$ is $t_0$-GIT-polystable by Proposition \ref{stackiso}.

    For the converse, suppose $(Q,[S])\in\mb{P}\mtc{E}$ is a $t_0$-GIT-semistable point. Then by Lemma \ref{unstablelocus}, we have that $(Q,[S])\in U$ and $C:=Q\cap S$ is a $(2,3)$-complete intersection curve in $\mb{P}^3$. Take $\{C_b\}_{b\in B}$ a family of $(2,3)$-complete intersection curves over a smooth pointed curve $(0\in B)$ such that $C_0=C$, $C_b$ is smooth and $\Bl_{C_b}\mb{P}^3$ is K-stable for any $b \in B \setminus \{0\}$. Then by the properness of K-moduli spaces, we have a K-polystable limit $X'_0$ of $\Bl_{C_b}\mb{P}^3$ as $b \to 0$, after possibly a finite base change of $B$. By Theorem \ref{23complete}, we know that $X'_0\simeq \Bl_{C_0'}\mb{P}^3$ for some $(2,3)$-complete intersection curve $C_0'$. By Proposition \ref{stackiso}, we know that $[C_b]$ admits a limit $[C_0']$ as $b\to 0$ in the stack $[U^K/\PGL(4)]$.  Since $[C_0']$ is $t_0$-GIT-polystable from earlier arguments, by the separatedness of GIT quotients we know that $C$ specially degenerates to $g\cdot C'_0$ for some $g \in\PGL(4)$. Thus $X$ is K-semistable by openness of K-semistability. 

    From the equivalence of K-semistability and $t_0$-GIT-semistability, we obtain the identification $U^K=U^{\sst}(t_0)$. Then the isomorphism of moduli stacks (resp. spaces) follows from Proposition \ref{stackiso}.
    
\end{proof}

\begin{lemma}\label{unstablelocus}\textup{(cf. \cite[Proposition 4.6]{CMJL14})} If $(Q,[S])\in\mb{P}\mtc{E}$ is a point which is not contained in $U$, then $(Q,[S])$ is GIT-unstable with respect to any slope $t\in[0,\frac{1}{2}]$.
\end{lemma}

\begin{proof}[Proof of Theorem \ref{Main1}]
    This follows from immediately from Theorem \ref{23complete}, Corollary \ref{cor:stack-smooth}, Theorem \ref{mainiso} and \cite[Main Theorem]{CMJL14}.
\end{proof}

As an immediate consequence of Theorem \ref{mainiso} and \cite[Theorem 7.1]{CMJL14}, one obtains the following result.

\begin{theorem}[= Theorem \ref{hkp}]
    The K-moduli space $\ove{M}^K_{\textup{№2.15}}$ is isomorphic to the model $\ove{M}_4(\frac{1}{2},\frac{23}{44})$ for the Hassett--Keel program of genus four curves.
\end{theorem}

\subsection{Classification of K-(semi/poly)stable objects}

For the reader's convenience, in this section we give a complete classification of K-(semi/poly)stable members in the family. This is a combination of Theorem \ref{mainiso} together with Theorem \ref{HassettK} and the classification of GIT-(semi/poly)stable curves in the linear series $|\mtc{O}_{\mb{P}^1\times\mb{P}^1}(3,3)|$ (cf. \cite[Section 2]{Fed12} and \cite[Lemma 5.17]{OSS16}).

\begin{theorem}[= Theorem \ref{Main2}]\label{23333}
    Let $X$ be the blow-up of $\mb{P}^3$ along a $(2,3)$-complete intersection curves $C$. Let $Q$ be the unique quadric surface in $\mb{P}^3$ containing $C$. Suppose that $X$ is K-semistable, then $Q$ is normal and $C$ is reduced. Moreover, the threefold $X$ is 
    \begin{enumerate}
        \item K-semistable if and only if one of the followings holds:
        \begin{itemize}
            \item $Q$ is a smooth quadric, and $C$ does not contain a line component $L$ meeting the residual curve $C'$ at a unique point, which is a singularity of $C'$. Equivalently, the curve $C$ has at worst $A_n$- ($n\leq 9$) or $D_4$-singularities;
            \item $Q$ is a quadric cone, and $C$ has at worst $A_n$- or $D_4$-singularities in the smooth locus of $Q$, and has at worst an $A_1$-singularity at the cone point. 
        \end{itemize}
        \item K-stable if and only if one of the followings holds:
         \begin{itemize}
            \item $Q$ is a smooth quadric, $C$ has at worst $A_n$-singularities, and $C$ contains no line component $L$ meeting the residual curve $C'$ in exactly one point;
            \item $Q$ is a quadric cone, $C$ has at worst $A_n$-singularities in the smooth locus of $Q$, and has at worst an $A_1$-singularity at the cone point. 
        \end{itemize}
        \item K-polystable but not K-stable if and only if one of the followings hold:
            \begin{itemize}
            \item $Q$ is normal and the curve $C$ is a union of three conics meeting in two $D_4$-singularities. These curves occupy a $2$-dimensional locus in the K-moduli;
             \item $Q\simeq\mb{V}(x_0x_3-x_1x_2)$ is a smooth quadric, and $C\simeq C_{2A_5}=\mb{V}(x_0x_3-x_1x_2,x_0x_2^2+x^2_1x_3)$ is the (unique) maximally degenerate curve with two $A_5$-singularities. 
        \end{itemize}
    \end{enumerate}
    In particular, every smooth member in this deformation family of Fano threefolds is K-stable.
\end{theorem}

\begin{remark}
    \textup{
    \begin{enumerate}
        \item The last statement of Theorem \ref{23333} was proved independently in \cite[Main Theorem]{GDGV23} using a completely different method. As a side remark, the approach in \cite{GDGV23} is developed in \cite{AZ22}, which is of computational flavor.
        \item It follows immediately from Theorem \ref{23333} that each K-semistable Fano variety in $\mts{M}^K_{\textup{№2.15}}$ has at worst $A_n$- or $D_4$-singularities. In particular, they are all Gorenstein terminal.
    \end{enumerate}}
\end{remark}

\bibliographystyle{alpha}
\bibliography{citation}

\newcommand{\etalchar}[1]{$^{#1}$}
\begin{thebibliography}{AFSvdW17}

\bibitem[ABHLX20]{ABHLX20}
Jarod Alper, Harold Blum, Daniel Halpern-Leistner, and Chenyang Xu.
\newblock Reductivity of the automorphism group of {K}-polystable {F}ano varieties.
\newblock {\em Inventiones mathematicae}, 222(3):995--1032, 2020.

\bibitem[ACC{\etalchar{+}}23]{CA21}
Carolina Araujo, Ana-Maria Castravet, Ivan Cheltsov, Kento Fujita, Anne-Sophie Kaloghiros, Jesus Martinez-Garcia, Constantin Shramov, Hendrik S{\"u}{\ss}, and Nivedita Viswanathan.
\newblock {\em {The Calabi problem for Fano threefolds}}, volume 485 of {\em London Mathematical Society Lecture Note Series}.
\newblock Cambridge University Press, Cambridge, 2023.

\bibitem[ACD{\etalchar{+}}]{ACD+23}
Hamid Abban, Ivan Cheltsov, Elena Denisova, Erroxe Etxabarri-Alberdi, Anne-Sophie Kaloghiros, Dongchen Jiao, Jesus Martinez-Garcia, and Theodoros Papazachariou.
\newblock {One-dimensional components in the K-moduli of smooth Fano 3-folds}.
\newblock {\em To appear in Journal of Algebraic Geometry}.

\bibitem[ACK{\etalchar{+}}]{ACKLP}
Hamid Abban, Ivan Cheltsov, Alexander Kasprzyk, Yuchen Liu, and Andrea Petracci.
\newblock {On K-moduli of quartic threefolds}.
\newblock {\em To appear in Algebraic Geometry}.

\bibitem[ADL23a]{ADL21}
Kenneth Ascher, Kristin DeVleming, and Yuchen Liu.
\newblock K-moduli of curves on a quadric surface and {K}3 surfaces.
\newblock {\em J. Inst. Math. Jussieu}, 22(3):1251--1291, 2023.

\bibitem[ADL23b]{ADL22}
Kenneth Ascher, Kristin DeVleming, and Yuchen Liu.
\newblock K-stability and birational models of moduli of quartic {K}3 surfaces.
\newblock {\em Invent. Math.}, 232(2):471--552, 2023.

\bibitem[ADL24]{ADL19}
Kenneth Ascher, Kristin DeVleming, and Yuchen Liu.
\newblock Wall crossing for k-moduli spaces of plane curves.
\newblock {\em Proceedings of the London Mathematical Society}, 128(6):12615, 2024.

\bibitem[AE23]{AE23}
Valery Alexeev and Philip Engel.
\newblock Compact moduli of {K3} surfaces.
\newblock {\em Annals of Mathematics}, 198(2):727--789, 2023.

\bibitem[AFS17a]{AFS173}
Jarod Alper, Maksym Fedorchuk, and David~Ishii Smyth.
\newblock {Second flip in the Hassett--Keel program: existence of good moduli spaces}.
\newblock {\em Compositio Mathematica}, 153(8):1584--1609, 2017.

\bibitem[AFS17b]{AFS172}
Jarod Alper, Maksym Fedorchuk, and David~Ishii Smyth.
\newblock {Second flip in the Hassett--Keel program: projectivity}.
\newblock {\em International Mathematics Research Notices}, 2017(24):7375--7419, 2017.

\bibitem[AFSvdW17]{AFS17}
Jarod Alper, Maksym Fedorchuk, David~Ishii Smyth, and Frederick van~der Wyck.
\newblock {Second flip in the Hassett--Keel program: a local description}.
\newblock {\em Compositio Mathematica}, 153(8):1547--1583, 2017.

\bibitem[Alp13]{Alp13}
Jarod Alper.
\newblock Good moduli spaces for {A}rtin stacks.
\newblock {\em Annales de l'Institut Fourier}, 63(6):2349--2402, 2013.

\bibitem[Amb99]{Amb99}
Florin Ambro.
\newblock Ladders on {F}ano varieties.
\newblock {\em Journal of Mathematical Sciences}, 94(1):1126--1135, 1999.

\bibitem[AZ22]{AZ22}
Hamid Abban and Ziquan Zhuang.
\newblock K-stability of {F}ano varieties via admissible flags.
\newblock {\em Forum Math. Pi}, 10:Paper No. e15, 43, 2022.

\bibitem[BCHM10]{BCHM}
Caucher Birkar, Paolo Cascini, Christopher~D. Hacon, and James McKernan.
\newblock Existence of minimal models for varieties of log general type.
\newblock {\em J. Amer. Math. Soc.}, 23(2):405--468, 2010.

\bibitem[Ben14]{Ben14}
Olivier Benoist.
\newblock Quelques espaces de modules d'intersections compl{\`e}tes lisses qui sont quasi-projectifs.
\newblock {\em Journal of the European Mathematical Society}, 16(8):1749--1774, 2014.

\bibitem[BHLLX21]{BHLLX21}
Harold Blum, Daniel Halpern-Leistner, Yuchen Liu, and Chenyang Xu.
\newblock {On properness of K-moduli spaces and optimal degenerations of Fano varieties}.
\newblock {\em Selecta Mathematica}, 27(4):73, 2021.

\bibitem[BL22]{BL18}
Harold Blum and Yuchen Liu.
\newblock Openness of uniform {K}-stability in families of {$\Bbb Q$}-{F}ano varieties.
\newblock {\em Ann. Sci. \'{E}c. Norm. Sup\'{e}r. (4)}, 55(1):1--41, 2022.

\bibitem[BLX22]{BLX19}
Harold Blum, Yuchen Liu, and Chenyang Xu.
\newblock Openness of {K}-semistability for {F}ano varieties.
\newblock {\em Duke Mathematical Journal}, 171(13):2753--2797, 2022.

\bibitem[BX19]{BX19}
Harold Blum and Chenyang Xu.
\newblock Uniqueness of {K}-polystable degenerations of {F}ano varieties.
\newblock {\em Annals of Mathematics}, 190(2):609--656, 2019.

\bibitem[CDGF{\etalchar{+}}23]{CDG+}
Ivan Cheltsov, Tiago Duarte~Guerreiro, Kento Fujita, Igor Krylov, and Jesus Martinez-Garcia.
\newblock {K-stability of Casagrande-Druel varieties}.
\newblock {\em arXiv preprint arXiv:2309.12522}, 2023.

\bibitem[CFFK24]{CFFK}
Ivan Cheltsov, Maksym Fedorchuk, Kento Fujita, and Anne-Sophie Kaloghiros.
\newblock K-moduli of pure states of four qubits.
\newblock {\em arXiv preprint arXiv:2412.19972}, 2024.

\bibitem[CMJL14]{CMJL14}
Sebastian Casalaina-Martin, David Jensen, and Radu Laza.
\newblock Log canonical models and variation of {GIT} for genus 4 canonical curves.
\newblock {\em Journal of Algebraic Geometry}, 23(4):727--764, 2014.

\bibitem[CP21]{CP21}
Giulio Codogni and Zsolt Patakfalvi.
\newblock Positivity of the {CM} line bundle for families of {K}-stable klt {F}ano varieties.
\newblock {\em Inventiones mathematicae}, 223(3):811--894, 2021.

\bibitem[CT23]{CT23}
Ivan Cheltsov and Alan Thompson.
\newblock {K-moduli of Fano threefolds in family 3.10}.
\newblock {\em arXiv preprint arXiv:2309.12524}, 2023.

\bibitem[DJKHQ24]{DJKQ}
Kristin DeVleming, Lena Ji, Patrick Kennedy-Hunt, and Ming~Hao Quek.
\newblock {The K-moduli space of a family of conic bundles threefolds}.
\newblock {\em arXiv preprint arXiv:2403.09557}, 2024.

\bibitem[Dol96]{Dol96}
I.~V. Dolgachev.
\newblock Mirror symmetry for lattice polarized {$K3$} surfaces.
\newblock {\em J. Math. Sci.}, 81(3):2599--2630, 1996.
\newblock Algebraic geometry, 4.

\bibitem[Don02]{Don02}
S.~K. Donaldson.
\newblock Scalar curvature and stability of toric varieties.
\newblock {\em J. Differential Geom.}, 62(2):289--349, 2002.

\bibitem[Fed12]{Fed12}
Maksym Fedorchuk.
\newblock The final log canonical model of the moduli space of stable curves of genus 4.
\newblock {\em International Mathematics Research Notices}, 2012(24):5650--5672, 2012.

\bibitem[FS13]{Fed13}
Maksym Fedorchuk and David~Ishii Smyth.
\newblock Stability of genus five canonical curves.
\newblock {\em A celebration of algebraic geometry}, 18:281--310, 2013.

\bibitem[Fuj90]{Fuj90}
Takao Fujita.
\newblock On singular del {P}ezzo varieties.
\newblock In {\em Algebraic geometry}, pages 117--128. Springer, 1990.

\bibitem[Fuj18]{Fuj18}
Kento Fujita.
\newblock Optimal bounds for the volumes of {K}{\"a}hler-{E}instein {F}ano manifolds.
\newblock {\em American Journal of Mathematics}, 140(2):391--414, 2018.

\bibitem[Fuj19]{Fuj19}
Kento Fujita.
\newblock A valuative criterion for uniform {K}-stability of $\mb{Q}$-{F}ano varieties.
\newblock {\em Journal f{\"u}r die reine und angewandte Mathematik (Crelles Journal)}, 2019(751):309--338, 2019.

\bibitem[GGV24]{GDGV23}
Tiago~Duarte Guerreiro, Luca Giovenzana, and Nivedita Viswanathan.
\newblock On {K}-stability of $\mathbb{P}^3$ blown up along a (2,3) complete intersection.
\newblock {\em Journal of the London Mathematical Society}, 110(1):e12961, 2024.

\bibitem[Has05]{Has05}
Brendan Hassett.
\newblock Classical and minimal models of the moduli space of curves of genus two.
\newblock In {\em Geometric methods in algebra and number theory}, pages 169--192. Springer, 2005.

\bibitem[HH09]{HH09}
Brendan Hassett and Donghoon Hyeon.
\newblock Log canonical models for the moduli space of curves: the first divisorial contraction.
\newblock {\em Transactions of the American Mathematical Society}, 361(8):4471--4489, 2009.

\bibitem[HH13]{HH13}
Brendan Hassett and Donghoon Hyeon.
\newblock Log minimal model program for the moduli space of stable curves: the first flip.
\newblock {\em Annals of Mathematics}, pages 911--968, 2013.

\bibitem[HL10]{HL10}
Donghoon Hyeon and Yongnam Lee.
\newblock Log minimal model program for the moduli space of stable curves of genus three.
\newblock {\em Math. Res. Lett.}, 17(4):625--636, 2010.

\bibitem[Huy16]{Huy16}
Daniel Huybrechts.
\newblock {\em Lectures on {K}3 surfaces}, volume 158.
\newblock Cambridge University Press, 2016.

\bibitem[IP99]{IP99}
V.~A. Iskovskikh and Yu.~G. Prokhorov.
\newblock Fano varieties.
\newblock In {\em Algebraic geometry, {V}}, volume~47 of {\em Encyclopaedia Math. Sci.}, pages 1--247. Springer, Berlin, 1999.

\bibitem[Jia20]{Jia20}
Chen Jiang.
\newblock Boundedness of $\mb{Q}$-{F}ano varieties with degrees and alpha-invariants bounded from below.
\newblock {\em Ann. Sci. Éc. Norm. Supér (4)}, 53(5):1235--1248, 2020.

\bibitem[Kaw00]{Kaw00}
Yujiro Kawamata.
\newblock On effective non-vanishing and base-point-freeness.
\newblock {\em Asian Journal of Mathematics}, 4(1):173--181, 2000.

\bibitem[KM76]{KM76}
Finn Knudsen and David Mumford.
\newblock The projectivity of the moduli space of stable curves i: Preliminaries on "det" and "div".
\newblock {\em Mathematica Scandinavica}, 39(1):19--55, 1976.

\bibitem[KM98]{KM98}
J{\'a}nos Koll{\'a}r and Shigefumi Mori.
\newblock {\em Birational geometry of algebraic varieties}, volume 134.
\newblock Cambridge university press, 1998.

\bibitem[KP21]{KP21}
Anne-Sophie Kaloghiros and Andrea Petracci.
\newblock On toric geometry and {K}-stability of {F}ano varieties.
\newblock {\em Trans. Amer. Math. Soc. Ser. B}, 8:548--577, 2021.

\bibitem[KSB88]{KSB88}
J{\'a}nos Koll{\'a}r and Nicholas~I Shepherd-Barron.
\newblock Threefolds and deformations of surface singularities.
\newblock {\em Inventiones mathematicae}, 91(2):299--338, 1988.

\bibitem[Li17]{Li17}
Chi Li.
\newblock K-semistability is equivariant volume minimization.
\newblock {\em Duke Mathematical Journal}, 166(16):3147--3218, 2017.

\bibitem[Liu18]{Liu18}
Yuchen Liu.
\newblock The volume of singular {K}{\"a}hler--{E}instein {F}ano varieties.
\newblock {\em Compositio Mathematica}, 154(6):1131--1158, 2018.

\bibitem[Liu22]{Liu22}
Yuchen Liu.
\newblock K-stability of cubic fourfolds.
\newblock {\em Journal f{\"u}r die reine und angewandte Mathematik (Crelles Journal)}, 2022(786):55--77, 2022.

\bibitem[LWX21]{LWX21}
Chi Li, Xiaowei Wang, and Chenyang Xu.
\newblock Algebraicity of the metric tangent cones and equivariant {K}-stability.
\newblock {\em Journal of the American Mathematical Society}, 34(4):1175--1214, 2021.

\bibitem[LX19]{LX19}
Yuchen Liu and Chenyang Xu.
\newblock K-stability of cubic threefolds.
\newblock {\em Duke Mathematical Journal}, 168(11):2029--2073, 2019.

\bibitem[LXZ22]{LXZ22}
Yuchen Liu, Chenyang Xu, and Ziquan Zhuang.
\newblock Finite generation for valuations computing stability thresholds and applications to {K}-stability.
\newblock {\em Annals of Mathematics}, 196(2):507--566, 2022.

\bibitem[May72]{May72}
Alan~L Mayer.
\newblock Families of {K3} surfaces.
\newblock {\em Nagoya Mathematical Journal}, 48:1--17, 1972.

\bibitem[MM93]{MM93}
Toshiki Mabuchi and Shigeru Mukai.
\newblock Stability and {E}instein-{K}\"{a}hler metric of a quartic del {P}ezzo surface.
\newblock In {\em Einstein metrics and {Y}ang-{M}ills connections ({S}anda, 1990)}, volume 145 of {\em Lecture Notes in Pure and Appl. Math.}, pages 133--160. Dekker, New York, 1993.

\bibitem[MM03]{MM03}
Shigefumi Mori and Shigeru Mukai.
\newblock Erratum: ``{C}lassification of {F}ano 3-folds with {$B_2\geq 2$}'' [{M}anuscripta {M}ath. {\bf 36} (1981/82), no. 2, 147--162; {MR}0641971 (83f:14032)].
\newblock {\em Manuscripta Math.}, 110(3):407, 2003.

\bibitem[M{\"u}l14]{Mul14}
Fabian M{\"u}ller.
\newblock The final log canonical model of $\mts{M}_6$.
\newblock {\em Algebra \& Number Theory}, 8(5):1113--1126, 2014.

\bibitem[Oda13]{Oda13}
Yuji Odaka.
\newblock The {GIT} stability of polarized varieties via discrepancy.
\newblock {\em Annals of Mathematics}, pages 645--661, 2013.

\bibitem[OSS16]{OSS16}
Yuji Odaka, Cristiano Spotti, and Song Sun.
\newblock Compact moduli spaces of del {P}ezzo surfaces and {K}{\"a}hler--{E}instein metrics.
\newblock {\em Journal of Differential Geometry}, 102(1):127--172, 2016.

\bibitem[Pap22]{Pap22}
Theodoros~Stylianos Papazachariou.
\newblock K-moduli of log {F}ano complete intersections.
\newblock {\em arXiv preprint arXiv:2212.09332}, 2022.

\bibitem[Pet21]{Pet21}
Andrea Petracci.
\newblock {K-moduli of Fano 3-folds can have embedded points}.
\newblock {\em arXiv preprint arXiv:2105.02307}, 2021.

\bibitem[Pet22]{Pet22}
Andrea Petracci.
\newblock On deformation spaces of toric singularities and on singularities of {K}-moduli of {F}ano varieties.
\newblock {\em Trans. Amer. Math. Soc.}, 375(8):5617--5643, 2022.

\bibitem[PT06]{PT06}
S~Paul and Gang Tian.
\newblock {CM} stability and the generalised {F}utaki invariant {I}.
\newblock {\em arXiv preprint arXiv:math/0605278}, 2006.

\bibitem[PT09]{PT09}
S.~T. Paul and G.~Tian.
\newblock {CM} stability and the generalized {F}utaki invariant {II}.
\newblock {\em Ast\'{e}risque}, 328:339--354, 2009.

\bibitem[Rei83]{Rei83}
Miles Reid.
\newblock {Projective morphisms according to Kawamata}.
\newblock {\em preprint}, 1983.

\bibitem[Rob76]{Rob76}
Lorenzo Robbiano.
\newblock Some properties of complete intersections in “good” projective varieties.
\newblock {\em Nagoya Mathematical Journal}, 61:103--111, 1976.

\bibitem[Sch71]{Sch71}
Michael Schlessinger.
\newblock Rigidity of quotient singularities.
\newblock {\em Inventiones mathematicae}, 14(1):17--26, 1971.

\bibitem[SD74]{SD}
B.~Saint-Donat.
\newblock Projective models of {$K-3$} surfaces.
\newblock {\em Amer. J. Math.}, 96:602--639, 1974.

\bibitem[Sho79]{Sho79}
V.~V. Shokurov.
\newblock Smoothness of a general anticanonical divisor on a {F}ano variety.
\newblock {\em Izv. Akad. Nauk SSSR Ser. Mat.}, 43(2):430--441, 1979.

\bibitem[SS17]{SS17}
Cristiano Spotti and Song Sun.
\newblock Explicit {G}romov-{H}ausdorff compactifications of moduli spaces of {K}\"{a}hler-{E}instein {F}ano manifolds.
\newblock {\em Pure Appl. Math. Q.}, 13(3):477--515, 2017.

\bibitem[{Sta}25]{Sta18}
The {Stacks Project Authors}.
\newblock {Stacks Project}, 2025.

\bibitem[Tia97]{Tia97}
Gang Tian.
\newblock K{\"a}hler-{E}instein metrics with positive scalar curvature.
\newblock {\em Inventiones mathematicae}, 130(1):1--37, 1997.

\bibitem[Xu20]{Xu20}
Chenyang Xu.
\newblock A minimizing valuation is quasi-monomial.
\newblock {\em Annals of Mathematics}, 191(3):1003--1030, 2020.

\bibitem[Xu21]{Xu21}
Chenyang Xu.
\newblock K-stability of {F}ano varieties: an algebro-geometric approach.
\newblock {\em EMS Surv. Math. Sci.}, 8(1):265--354, 2021.

\bibitem[Xu24]{Xu24}
Chenyang Xu.
\newblock {\em K-stability of {F}ano varieties}.
\newblock https://web.math.princeton.edu/~chenyang/Kstabilitybook.pdf, 2024.

\bibitem[XZ20]{XZ20}
Chenyang Xu and Ziquan Zhuang.
\newblock On positivity of the {CM} line bundle on {K}-moduli spaces.
\newblock {\em Annals of mathematics}, 192(3):1005--1068, 2020.

\bibitem[XZ21]{XZ21}
Chenyang Xu and Ziquan Zhuang.
\newblock Uniqueness of the minimizer of the normalized volume function.
\newblock {\em Camb. J. Math.}, 9(1):149--176, 2021.

\bibitem[Zha23]{zha23b}
Junyan Zhao.
\newblock The moduli space of genus six curves and {K}-stability: {VGIT} and the {H}assett-{K}eel program.
\newblock {\em arXiv preprint arXiv:2304.13259}, 2023.

\bibitem[Zhu21]{Zhu21}
Ziquan Zhuang.
\newblock Optimal destabilizing centers and equivariant {K}-stability.
\newblock {\em Inventiones mathematicae}, 226(1):195--223, 2021.

\end{thebibliography}

\end{document}